\documentclass[12pt]{amsart}

\title{Annular representations of free product categories}

\author[ \large
S\lowercase{hamindra} G\lowercase{hosh},
C\lowercase{orey} J\lowercase{ones},
M\lowercase{adhav} R\lowercase{eddy}
]
{\bf \large 
	S\lowercase{hamindra} K\lowercase{umar} G\lowercase{hosh},
	C\lowercase{orey} J\lowercase{ones and}
	B M\lowercase{adhav} R\lowercase{eddy}
}
\date{}
\address{Stat-Math Unit, Indian Statistical Institute, Kolkata, INDIA}
\email{shami@isical.ac.in}
\address{Australian National University, Mathematical Sciences Institute, Canberra, AUSTRALIA}
\email{cormjones88@gmail.com}
\address{Stat-Math Unit, Indian Statistical Institute, Kolkata, INDIA}
\email{madhav0903@gmail.com}

\usepackage{bbm}
\usepackage{mathtools}
\usepackage{cancel}
\usepackage{amsmath}
\usepackage{amsfonts}
\usepackage{latexsym}
\usepackage{amssymb}
\usepackage{mathrsfs}
\usepackage{amscd}
\usepackage{hyperref}
\usepackage{soul}
\usepackage{psfrag}
\usepackage{graphicx}
\usepackage{ulem}
\usepackage{fullpage}
\usepackage[all]{xy}
\usepackage{rotating}
\usepackage{url}
\usepackage{color}
\usepackage{enumerate}
\usepackage{cleveref}
\usepackage{dsfont}
\numberwithin{equation}{section}
\numberwithin{figure}{section}
\theoremstyle{plain}
\newtheorem{thm}{Theorem}[section]
\theoremstyle{plain}
\newtheorem{lem}[thm]{Lemma}
\theoremstyle{remark}
\newtheorem{rem}[thm]{Remark}
\theoremstyle{plain}
\newtheorem{cor}[thm]{Corollary}
\theoremstyle{definition}
\newtheorem{defn}[thm]{Definition}
\theoremstyle{definition}
\newtheorem{ex}[thm]{Example}
\theoremstyle{definition}

\theoremstyle{plain}
\newtheorem{prop}[thm]{Proposition}
\theoremstyle{plain}

\newcommand{\comments}[1]{}
\newcommand{\ra}{\rightarrow}

\newcommand{\mcal}{\mathcal}

\newcommand{\N}{\mathbb N}
\newcommand{\Z}{\mathbb Z}
\newcommand{\C}{\mathbb{C}}

\newcommand{\F}{\mathbb{F}}

\newcommand{\vlon}{\varepsilon}

\newcommand{\vphi}{\varphi}

\newcommand{\wt}{\widetilde}

\newcommand{\ID}{\textbf{I}_{\mcal D}}
\newcommand{\IC}{\textbf{I}_{\mcal C}}
\newcommand{\W}{\textbf{W}}
\newcommand{\Irr}{\text{Irr}}

\newcommand{\Hilb}{\textbf{Hilb}}

\begin{document}
	\global\long\def\vlon{\varepsilon}
	\global\long\def\bt{\bowtie}
	\global\long\def\ul#1{\underline{#1}}
	\global\long\def\ol#1{\overline{#1}}
	\global\long\def\norm#1{\left\|{#1}\right\|}
	\global\long\def\os#1#2{\overset{#1}{#2}}
	\global\long\def\us#1#2{\underset{#1}{#2}}
	\global\long\def\ous#1#2#3{\overset{#1}{\underset{#3}{#2}}}
	\global\long\def\t#1{\text{#1}}
	\global\long\def\lrsuf#1#2#3{\vphantom{#2}_{#1}^{\vphantom{#3}}#2^{#3}}
	\global\long\def\tr{\triangleright}
	\global\long\def\tl{\triangleleft}
	\global\long\def\cc90#1{\begin{sideways}#1\end{sideways}}
	\global\long\def\turnne#1{\begin{turn}{45}{#1}\end{turn}}
	\global\long\def\turnnw#1{\begin{turn}{135}{#1}\end{turn}}
	\global\long\def\turnse#1{\begin{turn}{-45}{#1}\end{turn}}
	\global\long\def\turnsw#1{\begin{turn}{-135}{#1}\end{turn}}
	\global\long\def\fusion#1#2#3{#1 \os{\textstyle{#2}}{\otimes} #3}
	
	\global\long\def\abs#1{\left|{#1}\right	|}
	\global\long\def\red#1{\textcolor{red}{#1}}

\maketitle

\begin{abstract}
We provide a description of the annular representation category of the free product of two rigid C*-tensor categories.

\end{abstract}

\section{Introduction}

Rigid C*-tensor categories have become important in recent years as descriptors of generalized symmetries appearing in noncommutative analysis and mathematical physics.
In operator algebras, they are closely connected to the standard invariants of finite index subfactors, and appear as the representation categories of compact quantum groups.
In the world of physics, they describe the superselection sectors in algebraic quantum field theories, and the structure of local excitations in 2 dimensional topological phases of matter.

An important algebra associated to a rigid C*-tensor category $\mcal{C}$ is the \textit{tube algebra} $\mcal{AC}$, first introduced by Ocneanu \cite{O}.
In the fusion case, this algebra has long been known as a useful tool for understanding the Drinfeld center (see \cite{I,Mu2}), while its importance in the case when $\mcal{C}$ has infinitely many simple objects has recently emerged.
The tube algebra admits a universal C*-algebra, hence has a well behaved representation category (see \cite{GJ}).
This category provides a useful way to describe the \textit{analytic properties of rigid C*-tensor categories}, such as amenability, the Haagerup property, and property (T).
These properties were first introduced by Popa in the context of subfactors (\cite{Po1,Po2}) and generalized to rigid C*-tensor categories by Popa and Vaes (\cite{PV}).
By \cite[Theorem 3.4]{PSV}, this category also provides a representation-theoretic characterization of the category $\mcal{Z}(\text{Ind-}\mcal{C})$, introduced by Neshveyev and Yamashita to provide a categorical understanding of analytic properties (\cite{NY1}).  
In a different direction, the annular representation theory of Temperley-Lieb-Jones categories has proved very useful in the classification of small index subfactor planar algebras (\cite{J2,JR,JMS}).

Unlike rigid C*-tensor categories themselves, whose underlying categorical structure is trivial due to semi-simplicity, the representation category of the tube algebra is a large W*-category, and is complicated to describe.
Thus an important problem is to find concrete descriptions of these large representation categories in terms of representation categories of more familiar C*-algebras such as group C*-algebras.
There are many procedures for producing new rigid C*-tensor categories from old ones, such as Deligne tensor product, equivariantization, $ G $-graded extensions, etc. 
A natural question is, if we understand the annular structure of our starting categories, can we describe the annular representation category of the one we have produced?

One such procedure is the free product construction, due to Bisch and Jones.
In this note, we will provide a decomposition of the category of annular representations of a free product category into the direct sum of four full W*-subcategories, where each component has an illuminating description.
To state the main result, first recall that $C^{*}_{u}(\mcal{C})$ is a C*-completion of the fusion algebra of $\mcal{C}$ with respect to admissible representations.
$Rep_{+}(\mcal{A}\mcal{C})$ denotes the full subcategory of annular representations which contain the fusion algebra, viewed as a corner of $\mcal{A}\mcal{C}$, in their kernels.
Finally, for two rigid C*-tensor categories $\mcal{C}$ and $\mcal{D}$, we let $\W$ be the set of words whose letters are alternatively taken from $\Irr(\mcal{C})\setminus \{[\mathbbm{1}]\}$ and $\Irr(\mcal{D}) \setminus \{[\mathbbm{1}]\}$ of even length with first letter coming from $\Irr(\mcal{C})\setminus \{[\mathbbm{1}]\}$.
We define an equivalence relation on the set $\W$ by $w_{1}\sim w_{2}$ if $w_{1}=uv$ and $w_{2}=vu$.  Set $\W_{0}:=\W/\sim$.
Then the main result of the paper is as follows:

\begin{thm}\label{mainthm}
Let $\mcal{C}$ and $\mcal{D}$ be rigid C*-tensor categories. Then as W*-categories,

$$ Rep(\mcal{A}\left(\mcal{C}*\mcal{D}\right))\cong Rep(C^{*}_{u}(\mcal{C})*C^{*}_{u}(\mcal{D}))\oplus Rep_{+}(\mcal{A}\mcal{C})\oplus Rep_{+}(\mcal{A}\mcal{D})\oplus Rep(\Z)^{\oplus \W_{0}} $$

\end{thm}

We remark that this decomposition is \textit{not} topological, in the sense that the first component $ Rep(C^{*}_{u}(\mcal{C})*C^{*}_{u}(\mcal{D}))$ is not necessarily closed.
In particular, it is possible to have a net of representations from $ Rep(C^{*}_{u}(\mcal{C})*C^{*}_{u}(\mcal{D}))$ converging to a representation in $Rep_{+}(\mcal{A}\mcal{C})\oplus Rep_{+}(\mcal{A}\mcal{D})$ in the Fell topology.  We do not pursue this here, but it would interesting to give a general characterization of the topology for free products.

The outline of the paper is as follows.
In the preliminaries section, we describe the construction of the free product of two rigid C*-tensor categories, as well as the basics of annular representation theory.
In the third section, we show that the representation category of the tube algebra is equivalent to the representation category of another annular algebra with a more convenient weight set.
We then provide a combinatorial analysis of the annular algebra vector spaces.
In the fourth section, we use these results to deduce the main result.
Finally we briefly discuss some examples.

\medskip

\noindent\textbf{Acknowledgements.} The authors would like to thank Dietmar Bisch, Mike Hartglass, David Penneys and Jyotishman Bhowmick for several useful discussions.
A part of this work was completed during the trimester program on von Neumann algebras (during May-Aug, 2016) at Hausdorff Research Institute for Mathematics (HIM) and the authors would like to thank HIM for the opportunity.
Corey Jones was supported by Discovery Projects ‘Subfactors and symmetries’ DP140100732 and ‘Low dimensional categories’ DP160103479 from the Australian Research Council.

\section{Preliminaries}\label{prelim}
\subsection{C*-tensor categories}

We refer the reader to \cite{NT,LR} for detailed definitions and basic results concerning C*-tensor categories, and in particular, rigid C*-tensor categories.

Recall that a $*$-structure on a $\C$-linear category $\mcal{C}$ is a conjugate linear functor from $*: \mcal{C}\rightarrow \mcal{C}^{op}$ which fixes objects, and satisfies $*\circ *=\textbf{Id}_{\mcal{C}}$.  Such a category will be called a $*$-category for short.
A \textit{C*-category} is a $*$-category, such that the morphism spaces $\mcal{C}(a,b)$ are equipped with Banach norms $||\cdot ||_{a,b}$ satisfying $||f^{*}f||_{a,a}=||ff^{*}||_{b,b}=||f||^{2}_{a,b}$ for all $f\in \mcal{C}(a,b)$.
Note that this makes each $\mcal{C}(a,a)$ into a C*-algebra, and we further require that $f^{*}f\ge 0$ in the C*-algebra $\mcal{C}(a,a)$ for all $f\in \mcal{C}(a,b)$. 
A W*-category is a C*-category such that each morphism space has a predual (\cite{GLR}).
We remark that although the norms on the spaces appear as additional structure, being a C* (or W*)-category is actually a property of a $ * $-category.
Indeed, one can take the semi-norms given by the spectral radius, and ask if they satisfy the conditions listed above.
In particular it makes sense to say a $ * $-category \textit{is} a C* (or W*)-category without specifying extra structure.

An object $a$ in a C*-category is called \textit{simple} if $\mcal{C}(a,a)=\C 1_{a}$, and $\mcal{C}$ is called \textit{semi-simple} if it has (unitary) direct sums, (self-adjoint) sub-objects, and every object is isomorphic to the direct sum of finitely many simple objects.
Note that in a semi-simple C*-category, all morphism spaces are finite dimensional, and so a semi-simple category is C* if and only if it is W*.

A \textit{C*-tensor category} is a C*-category equipped with a bilinear functor $\otimes: \mcal{C}\times \mcal{C}\rightarrow \mcal{C}$ together with unitary associativity natural transformations (called the associators) satisfying the pentagon axioms, and a distinguished unit object $\mathbbm{1}\in \mcal{C}$ with unitary unitor natural isomorphisms satisfying the triangle axioms.
By MacLane's strictness theorem we can (and usually do) assume our categories are strict, so that the associators and unitors are all identities.
In particular, this makes it easy to write tensor equations, and apply the usual graphical calculus formalisms.

A C*-tensor category \textit{has duals} if for every object $a\in \mcal{C}$, there exists an object $\overline{a}\in \mcal{C}$ and maps $R\in \mcal{C}(\mathbbm{1},\overline{a}\otimes a)$ and $\overline{R}\in \mathcal{C}(\mathbbm{1}, a\otimes \overline{a})$ satisfying the \textit{duality equations}:

$$(1_a\otimes R^{*})\circ (\overline{R}\otimes 1_{a})=1_{a}\ \text{and}\ (R^{*}\otimes 1_{\overline{a}})\circ (1_{\overline{a}}\otimes \overline{R})=1_{\overline{a}}$$

\begin{defn} A \textit{rigid C*-tensor category} is a semi-simple C*-tensor category with duals and simple tensor unit.
\end{defn}

Our definition of a rigid C*-tensor category is not universal, but is by far the most commonly studied.  

We recall here that for a semi-simple tensor category $\mcal{C}$, the \textit{fusion algebra} is the complex linear span of isomorphism classes of simple objects, with product given by the linear extension of $[a]\cdot [b]:=\sum_{c\in \Irr(\mcal{C})} N^{c}_{ab}[c]$, where $N^{c}_{ab}=\dim\left(\mcal{C}(a\otimes b, c)\right)$.
This is an associative, unital algebra.
When $\mcal{C}$, in addition, is rigid, there is a $*$-structure on this algebra, given by the conjugate linear extension of $[a]^*:=[\overline{a}]$.
This associative $*$-algebra is denoted $\text{Fus}(\mcal{C})$.

We have the following large class of (not mutually exclusive) examples, which indicate their connections with other areas of mathematics and physics:

\begin{enumerate}
\item
The category of bifinite Hilbert space bimodules of an $\rm{II}_1$ factor.
\item
$Rep(\mathbbm{G})$ for $\mathbbm{G}$ a compact (quantum) group.
\item
The DHR category of a covariant net of von Neumann algebras.
\end{enumerate}

The first example is actually universal, in the sense that every (countably generated) rigid C*-tensor category arises as a full subcategory of bimodules of the group von Neumann algebra $L\mathbbm{F}_{\infty}$ (\cite{PS,BHP}).

\subsection{Free product of categories}\label{free prod prel}
In this subsection, we provide the definition of the free product of two semi-simple C*-tensor categories with simple tensor units.
This notion, due to Bisch and Jones, arises from the free composition of finite index subfactors (see \cite{BJ}). It also appears in the study of free products of compact quantum groups \cite{W}.
Our approach to free products closely follows the construction of Bisch and Jones as elaborated by \cite{IMP}, except we do not require duals in our categories.

To proceed with this construction, we will first define a certain C*-category involving the two given categories, and controlled by non-crossing partitions. 
The free product will be the resulting projection category.

Let $ \mcal C_+ $ and $ \mcal C_- $ be two semi-simple C*-categories with simple tensor units $ \mathbbm {1}_+ $ and $ \mathbbm {1}_- $ respectively.
In our construction, we pick a strict model of $\mcal{C}_{\pm}$.
Let $ \Sigma $ be the set of words with letters in $ \t {Obj } \mcal C_+ \cup \t {Obj }  \mcal C_-$.
For $ \sigma \in \Sigma $, the length of $ \sigma $  will be denoted by  $ \abs \sigma $.
To a word $ \sigma \in \Sigma$, we associate the subword (whose letters are not necessarily adjacent) $ \sigma_+ \in \t {Obj }\mcal C_+ $ (resp., $ \sigma_- \in \t {Obj } \mcal C_- $) consisting of all the letters in $ \sigma $ coming from $ \t {Obj } \mcal C_+ $ (resp., $\t {Obj } \mcal C_- $).
The object obtained by tensoring the letters in $ \sigma_\pm $ will be denoted by $ t( \sigma_\pm )$ with the convention $ t(\emptyset) = \mathbbm {1_\pm}$ where appropriate.
For instance, if $ \sigma = a^+_1 a^-_2 a^+_3 a^-_4 a^-_5 $, then $ \sigma_+ =  a^+_1 a^+_3$, $t (\sigma_+) = a^+_1 \otimes a^+_3$, $ \sigma_- = a^-_2 a^-_4 a^-_5 $ and $t(\sigma_-) = a^-_2 \otimes a^-_4 \otimes a^-_5 $.
\begin{defn}
	Let $ \sigma, \tau \in \Sigma $.
A `\textit{$ (\sigma,\tau )$-NCP}' consists of:
\begin{itemize}
\item  a \textit{non-crossing partitioning} of the letters in $ \sigma $ and $ \tau $ arranged at the bottom and on the top edges of a rectangle respectively moving from left to right, such that each partition block consists only of objects from $ \mcal C_+ $ or only of objects $ \mcal C_- $ ,
%\item to every partition $ \{\sigma_0,\tau_0\} $ with the bottom labelled by $\sigma_0$  and the top labelled by $ \tau_0  $, where $ \sigma_0 $ and $ \tau_0 $ are subwords of $ \tau $ and $ \sigma $ respectively, we assign a morphism from $t(\sigma_0)  $ to $ t(\tau_0) $ in the appropriate category.
\item every block gives a pair of (possible empty) subwords of $ \sigma $ and $ \tau $, say, $ (\sigma_0,\tau_0) $, where $ \sigma_0 $ (resp. $ \tau_0 $) consists of letters in the partition coming from $ \sigma $ (resp. $ \tau $).
%Whenever $ \sigma_0 $ or $ \tau_0 $ is the empty subword, we take it to be the tensor unit of the appropriate category.
For each such block, seen as a rectangle with the bottom labeled by $\sigma_0$ and the top labeled by $ \tau_0  $, we choose a morphism from $t(\sigma_0)  $ to $ t(\tau_0) $ in the appropriate category.

\end{itemize}
\end{defn}
We give an example of a $ (\sigma , \tau) $-NCP in Figure \ref{NCPeg} where $ \sigma = a^+_1 a^+_2 a^+_3 a^-_4 a^+_5a^+_6 a^-_7 a^+_8  $ and $ \tau =  b^+_1 b^-_2 b^-_3 b^+_4 b^+_5 $ with $ a_i^\vlon,b_{j}^\vlon \in \mcal C_\vlon$, $ \vlon \in \{+,-\} $.

\begin{figure}[h]
\psfrag{1}{$ a^+_1 $}
\psfrag{2}{$ a^+_2 $}
\psfrag{3}{$ a^+_3 $}
\psfrag{4}{$ a^-_4 $}
\psfrag{5}{$ a^+_5 $}
\psfrag{6}{$ a^+_6 $}
\psfrag{7}{$ a^-_7 $}
\psfrag{8}{$ a^+_8 $}
\psfrag{f}{$ f_1 $}
\psfrag{g}{$ f_2 $}
\psfrag{h}{$ f_3 $}
\psfrag{i}{$ f_4 $}
\psfrag{j}{$ f_5 $}
\psfrag{A}{$ b^+_1 $}
\psfrag{B}{$ b^-_2 $}
\psfrag{C}{$ b^-_3 $}
\psfrag{D}{$ b^+_4 $}
\psfrag{E}{$ b^+_5 $}
\includegraphics[scale=0.3]{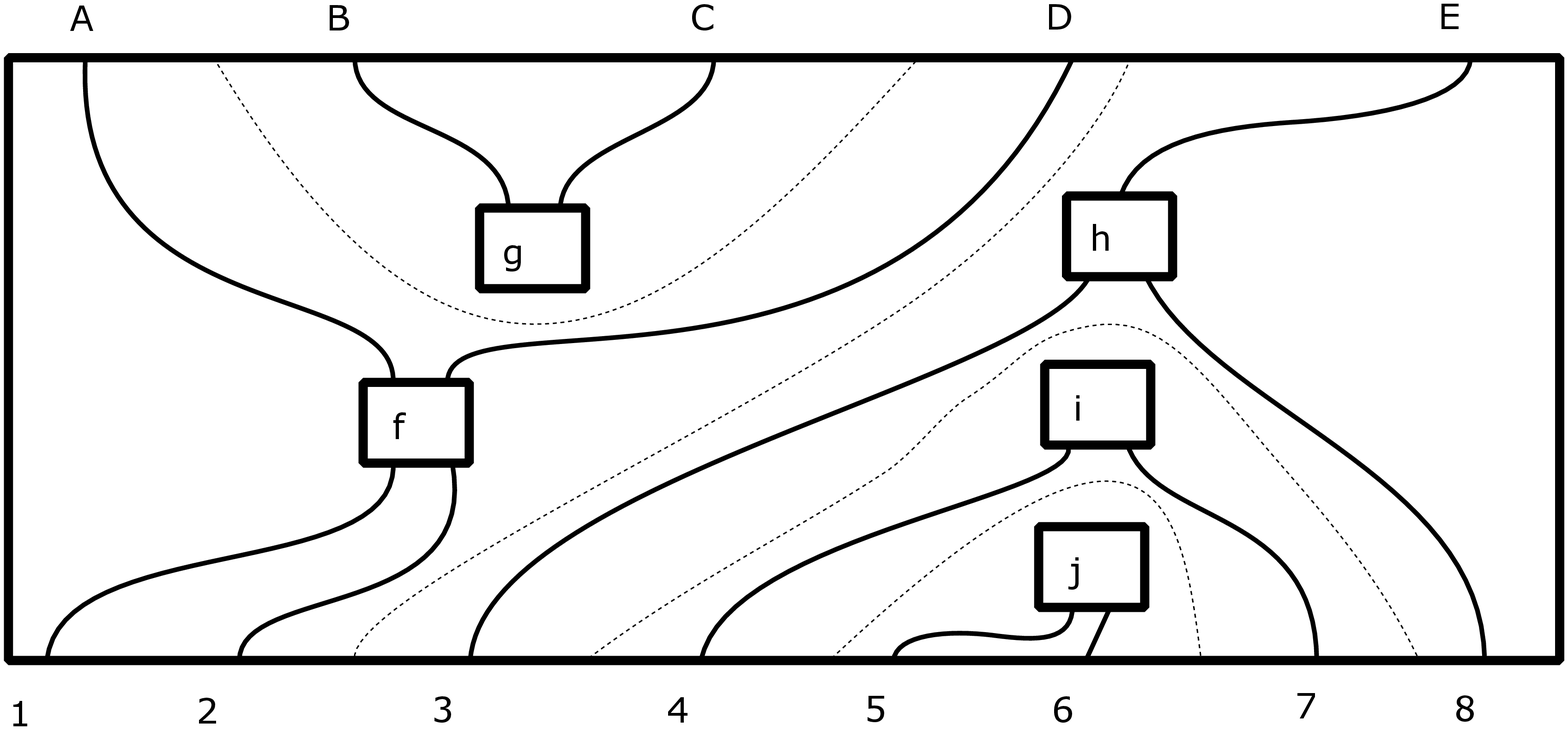}
\caption{}
\label{NCPeg}
\end{figure}
Here, the pair of subwords corresponding to the partition blocks are $ \rho_1 = (a_1^+a_2^+,b_1^+b_4^+), \rho_2=(\emptyset,b_2^-b_3^-),\rho_3=(a_3^+a_8^+,b_5^+),\rho_4=(a_4^-a_7^-,\emptyset),\t{and }\rho_5= (a_5^+a_6^+,\emptyset) $.
Note that each letter of $ \rho_i$ either belongs Obj $ \mcal C_+ $ alone or Obj $ \mcal C_- $ alone, for every $ i $ and each of $ \rho_i $ is assigned a morphism from the corresponding category.
For instance, all the letters of $ \rho_3 $ are objects of $ \mcal C_+ $ and is assigned the morphism $ f_3 \in \mcal C_+(a_3^+\otimes a_8^+,b_5^+) $.
%Each of these partitions, $ \rho_i \subset \t{Obj}(\mcal C_\vlon)$, are assigned by the morphisms $ a_i$ in $ \mcal C_\vlon $ for $ \vlon \in \{+,-\} $. 
\vskip 1em
We denote the set of such $ (\sigma , \tau )$-NCPs by $ NCP(\sigma,\tau) $.
Now, to every $ T \in NCP(\sigma , \tau ) $, we can associate unique $ T_{\pm}\in NCP(\sigma_{\pm} , \tau_{\pm} )$  by deleting all blocks whose letters are labeled by the opposite sign.
%For instance, if $ T $ denotes the NCP in \red{Figure ??}, then $ T_\pm $ is given by:
%\red{Figure pending}

Since all letters in $ \sigma_{\pm} $ and $ \tau_{\pm} $ come from either $ \mcal C_+ $ or $ \mcal C_- $ only, the non-crossing partitions  $T_{\pm} $ give rise to unique morphisms $ Z_{T_\pm}\in \mcal C_\pm \left(t(\sigma_\pm) , t(\tau_\pm) \right) $ using the standard graphical calculus for monoidal categories.

%obtained by tensoring the morphisms which are assigned to the partitions starting left to right.
%Arranging the partitions from left to right causes ambiguity whenever there is a partition which stays fully on the top or at the bottom.
%However, strictness of the categories $ \mcal C_\pm $ yeilds in a unique $ Z_T $ no matter how we sequence the partitions.
So, for any $ \sigma ,\tau \in \Sigma $ and $ T \in NCP (\sigma  , \tau) $, we have morphisms $ Z_{T_\pm} \in \mcal C_\pm \left(t(\sigma_\pm) , t(\tau_\pm) \right) $.
We write $ Z_T\coloneqq Z_{T_+} \otimes Z_{T_-} \in \mcal C_+ \left(t(\sigma_+), t(\tau_+)\right) \otimes \mcal C_- \left(t(\sigma_-), t(\tau_-)\right)$.
For example, for the NCP $ T $ in Figure \ref{NCPeg}, 
$$ Z_{T_+} = (f_1\otimes f_3)\circ (1_{a_1^+\otimes a_2^+\otimes a_3^+}\otimes f_5 \otimes1_{a_8^+})\ \text{and}\ Z_{T_+} = f_2\circ f_4$$

We define the category $ \mcal{NCP} $ as follows:

\begin{itemize}
\item Objects in $\mcal{NCP}$ are given by $\Sigma$.

\bigskip

\item For $ \sigma , \tau \in \Sigma $, the morphism space is defined by
$$\mcal {NCP} (\sigma , \tau) \coloneqq \t {span} \left\{Z_T: T \in NCP(\sigma , \tau) \right\} \subset \mcal C_+ \left(t(\sigma_+), t(\tau_+)\right) \otimes \mcal C_- \left(t(\sigma_-), t(\tau_-)\right).$$

\end{itemize}

Composition of morphisms is given by composing the tensor components, which is obviously bilinear, and associative.
However, one needs to verify whether the morphism spaces of $ \mcal{NCP} $ are closed under such composition.
Let $ S \in NCP (\sigma , \tau) $ and $ T \in NCP (\tau , \kappa) $.
Consider the `composed' rectangle obtained by gluing $ T $ on the top of $ S $ matching along the letters of $ \tau $.
The non-crossing partitions of $ S $ and $ T $ induce a non-crossing partition on the composed rectangle with $ \sigma  $ at the bottom and $ \tau $ on the top; each partition is then labeled by composing the corresponding morphisms in $ S $ and $ T $.
We call this $ T \circ S \in NCP(\sigma , \kappa) $.
In this process of composing two NCPs, we have ignored certain partitions of $ S $ (staying only on its top) and $ T $ (staying only at its bottom) which cancel each other and do not contribute towards the non-crossing partitioning of the composed rectangle.
Since the tensor units $ \mathbbm {1}_\pm $ are assumed to be simple, composing the morphisms associated to these partitions simply yield a scalar.
Suppose $ \lambda (T,S) $ denote the product of all such scalars.
Then, $ \left(Z_{T_+} \circ Z_{S_+}\right) \otimes \left(Z_{T_-} \circ Z_{S_-}\right) = \lambda (T,S) \; Z_{(T\circ S )_+} \otimes Z_{(T\circ S)_-} \in \mcal {NCP} (\sigma , \kappa)$.

Clearly, $ \mcal{NCP} $ is a $ \C $-linear category.
There is also a $ * $-structure given by applying $ * $ on each of the tensor components.
To see whether the morphism spaces of $ \mcal{NCP} $ is closed under $ * $, we define an involution $ \left( NCP(\sigma , \tau) \ni T \os {\displaystyle *} \longmapsto T^* \in NCP(\tau,\sigma) \right)_{\sigma,\tau \in \Sigma} $ where we reflect $ T $ about any horizontal line  to obtain $ T^* $ with a non-crossing partitioning and their corresponding morphisms being induced by the reflection of the initial partitioning and $ * $ of the assigned morphisms in $ T $ respectively.

Indeed, $ Z^*_T = Z_{T^*} \in \mcal{NCP} (\tau , \sigma)  $ for all $ T \in NCP(\sigma , \tau) $.
Thus, $ \mcal{NCP} $ is a $ * $-category.
Note that by construction, $\mcal{NCP}$ is equipped with a canonical faithful $*$-functor to the Deligne tensor product $\mcal{C}_{+}\boxtimes \mcal{C}_{-}$, which sends $\sigma$ to $\sigma_{+}\boxtimes \sigma_{-}\in \mcal{C}_{+}\boxtimes \mcal{C}_{-}$.
Since $\mcal{C}_{\pm}$ are both semi-simple, the Deligne tensor product is again a C*-category with finite dimensional morphism spaces.
But any (not necesarily full) *-subcategory of a C*-category with finite dimensional morphism spaces is easily seen to be C* itself.
Since our canonical functor is faithful, this implies $\mcal{NCP}$ is a C*-category.

For the tensor structure, define $ \sigma \otimes \tau $ as the concatenated word $ \sigma\tau$.
If $ f = \sum \limits_i a_i \otimes b_i \in \mcal{NCP} (\sigma , \tau) \subset \mcal C_+(\sigma_+, \tau_+) \otimes \mcal C_-(\sigma_-, \tau_-) $ and $ g = \sum \limits_j c_j \otimes d_j \in \mcal{NCP} (\kappa , \nu) \subset \mcal C_+(\kappa_+, \nu_+) \otimes \mcal C_-(\kappa_-, \nu_-)$, then $ f \otimes g \coloneqq \sum \limits_{i,j} (a_i \otimes^+ c_j) \otimes (b_i \otimes^- d_j )$ where $ \otimes^\pm $ denote for the tensor functor of $ \mcal C^\pm $.
It is easy to check $ f \otimes g \in \mcal{NCP} (\sigma \otimes \kappa , \tau \otimes \nu) $ and $ (f\otimes g )^* = f^* \otimes g^* $.
This implies $ \mcal{NCP} $ is a C*-tensor category.
Note that $ \mcal C_\pm $ sit inside $ \mcal{NCP} $ as full $ * $-subcategories.

We now define $ \mcal C_+ \ast \mcal C_-  $ to be the projection category of $\mcal{NCP}$.  More explicitly, 

$$\t {Obj } (\mcal C_+ \ast \mcal C_-) := \{(\sigma, p)\ :\ \sigma\in \Sigma\ \text{and}\ p\in \mcal{NCP}(\sigma,\sigma),\ \ p^{2}=p^{*}=p\}$$

 For $ (\sigma, p), (\tau, q)\in \t {Obj } (\mcal C_+ \ast \mcal C_-)  $, the morphism space 

 $$ (\mcal C_+ \ast \mcal C_-  )((\sigma,p),(\tau,q)) \coloneqq q \circ \mcal {NCP} (\sigma,\tau) \circ p$$
The tensor and $ * $-structures are induced by those of $ \mcal {NCP} $ in the obvious way.
It is easy to see that $ \mcal C_+ \ast \mcal C_-   $ is also C*-tensor category.

\begin{defn}
Let $\Irr(\mcal{C}_{\pm})$ denote a choice of object from each isomorphism class of simple objects, such that the tensor units are chosen to represent their isomorphism class.
Then $\Sigma_{0}:= \{\emptyset\} \cup \left\{a^{\vlon_{1}}_{1} \ldots  a^{\vlon_{k}}_{k} \; : \;  k\in \N, \ \vlon_{i} \in \{ \pm \} , \ a^{\vlon_i}_{i}\in \Irr(\mcal{C}_{\vlon_i})\setminus\{\mathbbm{1}_{\vlon_i}\},\ \vlon_{i} = -\vlon_{i+1} \t { for } 1\leq i \leq k \right\} $.
\end{defn}

\begin{prop}
$ \mcal C_+ \ast \mcal C_-  $ is a semi-simple C*-tensor category containing $ \mcal C_\pm $ as full subcategories.
Moreover, there is a canonical bijection between $\Sigma_{0}$ and isomorphism classes of simple objects in $\mcal{C}_{+}*\mcal{C}_{-}$, given by $\sigma \mapsto [ (\sigma, 1_{\sigma})]$.
\end{prop}
\begin{proof}
First we show that the objects $\sigma\in \Sigma_{0}$ form a distinct set of irreducible objects in $\mcal{NCP}$.

Let $ \sigma \in \Sigma_0 $ be a nonempty word, and $ T $ a $ (\sigma,\sigma) $-NCP.
If $ T $ has a block which connects only letters on the top or only letters on bottom, then $ T $ necessarily also has a singleton block and its associated morphism turns out to be zero (since $ \sigma \in \Sigma_0 $ is non-empty and the tensor units $ \mathbbm{1}_\pm $ are simple) which implies $ Z_T = 0 $.
Thus every partition in $ T $ consists of letters in the top as well as bottom.
Since the letters in $ \sigma $ come alternatively from $ \mcal C_+ $ and $ \mcal C_- $, and the partitions are non-crossing, the partition blocks should be of the form $ (a^{bottom}_1, a^{top}_1) $, $ (a^{bottom}_2, a^{top}_2) , \ldots $, where $ \sigma = a_{1}a_{2}\dots $.
The assigned morphisms of these blocks are then scalars since $ a_i $'s are simple.
This says that $ Z_T $ has to be a scalar multiple of $ 1_\sigma $.
Hence, $ \mcal{NCP(\sigma,\sigma)} $ is one-dimensional implying $ \sigma $ is simple for all $ \sigma \in \Sigma_0 $.
Similar arguments will tell us that $ \mcal{NCP(\sigma , \tau)} $ is zero for two distinct $ \sigma , \tau \in \Sigma_0 $.

We now show $\Sigma_{0}$ is complete, in the sense that any object $\sigma\in \mcal{NCP}$ is isomorphic to a direct sum of objects from $\Sigma_{0}$.
Observe that if $ \sigma_1, \ldots , \sigma_n \in \Sigma$ such that the letters in each $ \sigma_i $ come from $ \mcal C_+ $ alone or $ \mcal C_- $ alone, then $ \sigma_1 \ldots  \sigma_n$ is isomorphic to the word $ t(\sigma_1) \ldots  t(\sigma_n) $.  Moreover, a quick sketch of non-crossing partitions shows that $ \sigma  \mathbbm{1}_\pm  \tau \cong \sigma \tau $.
It is also easy to see that if $a\cong b_{1}\oplus b_{2}$ in $\mcal{C}_{\pm}$ via decomposition isometries $v_i\in \mcal{C}_{\pm}(b_{i}, a)$ , then the word $ \sigma a \tau \cong \sigma b_{1} \tau  \bigoplus \sigma b_{2} \tau  $ via decomposition isometries given by the $(\sigma b_{1} \tau, \sigma a \tau)$-non-crossing partitions $T_{i}$ defined as follows:  The underlying non-crossing partition has pairings which connect elements vertically, and for each block ending in $\sigma$ or $\tau$, we have the identity morphism, while the block connecting $b_{i}$ with $a$ is assigned the isometry $v_{i}$.  Taken together, these observations imply that any object can be decomposed as a finite direct sum of words in $ \Sigma_0 $.

\vskip 1em
$ \mcal C_+\ast \mcal C_- $ inherits all the above properties from $ \mcal{NCP} $.
In particular, since every object $\sigma\in \mcal{NCP}$ is isomorphic to a direct sum of simple objects in $\Sigma_{0}$, this will be true for any subobject.
Hence in the projection category, every object $(\tau, p)$ is isomorphic to a direct sum of objects of the form $(\sigma, 1_{\sigma})$ for $\sigma\in \Sigma_{0}$.

Thus to show that $\mcal C_+\ast \mcal C_-$ has direct sums,
it suffices to show that for $\sigma,\tau\in \Sigma_{0}$, there exists an object $(\sigma, 1_{\sigma})\oplus (\tau, 1_{\tau})\in \mcal C_+\ast \mcal C_- $ satisfying direct condition.

%We need to see $ p\oplus q $ as a projection in $ \mcal{NCP}(\gamma,\gamma) $ for some $ \

Let $ \alpha_i $ and $ \vlon_j $ be the signs given by $ a_i \in \mcal C_{\alpha_i} $ and $ b_j \in \mcal C_{\vlon_j}$.
Consider
$\widehat a_i \coloneqq a_i \oplus \mathbbm{1}_{\alpha_i} \text { implemented by the isometries }\ u_i \in \mcal C_{\alpha_i} (a_i , \widehat a_i) \text{ and }\ e_i \in \mcal C_{\alpha_i} (\mathbbm{1}_{\alpha_i } , \widehat a_i) $. Similarly, pick $\widehat b_j \coloneqq b j \oplus \mathbbm{1}_{\vlon_j} \text { and implementing isometries }\ v_j \in \mcal C_{\vlon_j} (b_j , \widehat b_j) \text { and }\ f_j \in \mcal C_{\vlon_j} (\mathbbm{1}_{\vlon_j} , \widehat b_j).$
Set 

$$
\left\{
\begin{tabular}{l}
$\widehat \sigma \coloneqq \widehat a_1  \ldots  \widehat a_m\ \text{and}$\ 
 $\sigma^{\prime} \coloneqq a_1 \ldots  a_m \mathbbm 1_{\vlon_1} \ldots \mathbbm 1_{\vlon_n}$,\\
  $\widehat \tau \coloneqq \widehat b_1  \ldots \widehat b_n\ \text{and}\ \tau^{\prime} \coloneqq \mathbbm 1_{\alpha_1} \ldots \mathbbm 1_{\alpha_m} b_1 \ldots b_n $ ,\\
 $\gamma \coloneqq \widehat \sigma  \widehat \tau $.
\end{tabular}
\right.
$$

We have already seen that $  \sigma^{\prime} \cong \sigma $ and $  \tau^{\prime} \cong \tau $ in $ \mcal{NCP} $.
Consider the isometries $ u \coloneqq u_1 \otimes \cdots \otimes u_m \otimes  \mathbbm{1}_{\vlon_1} \otimes \cdots \otimes \mathbbm{1}_{\vlon_n} \in \mcal{NCP}(\sigma^{\prime} , \gamma)$ and $ v \coloneqq \mathbbm{1}_{\alpha_1} \otimes \cdots \otimes \mathbbm{1}_{\alpha_m} \otimes v_1 \otimes \cdots \otimes  v_n \in \mcal{NCP} (\tau^{\prime}, \gamma)$.
Note that projections $ uu^* $ and $ vv^* $ are mutually orthogonal in $ \mcal NCP (\gamma, \gamma) $ (since $ (u_i u^*_i \;, \;e_i e^*_i )$ and $ (v_j v^*_j\; ,\; f_j f^*_j) $ are pairs of mutually orthogonal projections).
So, we have a projection in $\mcal {NCP} (\gamma , \gamma) $, namely $(uu^* + vv^*) \cong 1_{\sigma^{\prime}} \oplus 1_{ \tau^{\prime}} \cong 1_\sigma \oplus 1_\tau$ in $ \mcal C_+ \ast \mcal C_-$.
%In particular, $ p \oplus q \cong Z_upZ_u^* + Z_vqZ_v^* \in \mcal{NCP(\gamma,\gamma)} $ and we are done with the proof.
%$ u \coloneqq u_1 \otimes \cdots \otimes u_m \otimes f_1 \otimes \cdots \otimes f_n $ and  $ v \coloneqq e_1 \otimes \cdots \otimes e_m \otimes v_1 \otimes \cdots \otimes v_n $.
%Since the morphism spaces of $ \mcal{NCP} $ are finite dimensional, the endomorphism spaces become finite dimensional $ C^* $-algebras.
%So, any projection can be expressed as a finite sum of mutually orthogonal minimal projections of the corresponding endomorphism space of $ \mcal C_+ \ast \mcal C_- $.
%Now, minimal projections of these endomorphism spaces are precisely the simple objects.
%Next, observe that if $ \sigma_1, \ldots , \sigma_n \in \Sigma$ such that the letters in each $ \sigma_i $ come from $ \mcal C_+ $ alone or $ \mcal C_- $ alone, then $ (\sigma_1, \ldots , \sigma_n) $ is isomorphic to $ (t(\sigma_1) , \ldots , t(\sigma_n) ) $ in $ \mcal {NCP} $; moreover, if $ X \cong Y \bigoplus Z $ in $ \mcal C_\pm $, then $ (\sigma, X, \tau) \cong (\sigma,Y,\tau) \bigoplus (\sigma,X,\tau) $ in $ \mcal {NCP} $.
%This implies that every 
\end{proof}

\subsection{Annular representations}
Now we recall from \cite{GJ}, the definition and basic properties of \textit{annular algebras}, and their \textit{representation categories} associated to a rigid C*-tensor category $\mcal{C}$.
For a simple object $a\in \mcal{C}$ and an arbitrary object $b\in \mcal{C}$, we naturally have an inner product on $ \mcal{C}(a,b) $ given by $ g^*f=\langle f,g \rangle 1_a $.
Let $ \Irr(\mcal C)$ denote a set of representatives of isomorphism classes of simple objects in $ \mcal C $.
We assume that $ \mathbbm{1} \in \Irr(\mcal C) $ is chosen to represent its isomorphism class.
Let $ \Lambda $ be any subset of the set representatives of isomorphism classes of all objects in $ \mcal C $.
Then the \textit{annular algebra with weight set $ \Lambda $} is defined as a vector space

$$ \mcal A\Lambda:= \bigoplus\limits_{b,c\in \Lambda, a\in \t{Irr}(\mcal C)} \mcal{C}(a\otimes b , c\otimes a )$$.

For $ f\in \mcal{C}(a_1 \otimes b_1 ,b_2 \otimes a_1 ) $ and $ g\in \mcal C(a_2 \otimes b_3 ,b_4 \otimes a_2 ) $, multiplication in $ \mcal A\Lambda $ is given by 
\[
f \cdot g \; := \; \delta_{b_1=b_4} \; \sum_{c\in \Lambda} \; \sum_{u \in \t {onb} (\mcal{C} (c,a_1\otimes a_2 ))} (1_{b_2} \otimes u^{*}  ) (f \otimes 1_{b_2}  ) (1_{a_1} \otimes g  ) (u \otimes 1_{b_3}  )
\]
where $ onb $ denotes an orthonormal basis with respect to the inner product defined above.
This multiplication is associative and is independent of choice of representatives of isomorphism classes of simple objects and choice of $ onb $ in consideration.
$ \mcal A\Lambda $ has a $ * $-structure, which we denote by $ \# $, defined by
\[
f^\#:=(R_a^* \otimes 1_{b_1} \otimes 1_{\bar{a}} )(1_{\bar{a}}\otimes f^* \otimes 1_{\bar{a}})(1_{\bar{a}} \otimes 1_{b_2} \otimes \bar{R}_a )
\]
for $ f \in \mcal C(a \otimes b_1, b_2 \otimes a ) $.
The associative $ * $-algebra $ \mcal A\Lambda $ is unital if and only if $\Irr(\mcal{C})<\infty$.
This algebra has a canonical trace defined by $ \Omega(f):= \delta_{b=c} \; \delta_{a=\mathbbm{1}} \; tr(f) $ for all $ f \in \mcal C(a \otimes b, c \otimes a ) $, where $ tr $ is the unnormalized categorical trace on $ \mcal{C}(b,b) $, $tr(f):=R^{*}_{b}(1_{\overline{b}}\otimes f) R_{b}=\overline{R}^{*}_{b}(f\otimes 1_{b}) \bar{R}_{b}$.

We denote the subspaces 

$$ \mcal A\Lambda_{b,c}^{a}:= \mcal{C}(a\otimes b, c\otimes a) \subset \mcal A\Lambda\ \text{and}\ \mcal A\Lambda_{b,c}:= \bigoplus\limits_{a\in \Irr(\mcal C)} \mcal A\Lambda_{b,c}^a \subseteq \mcal{A}\Lambda$$

The associative $*$-algebra $\mcal A\Lambda _{b,b} $ is called the \textit{weight $ b $ centralizer algebra}.
We call $ \mcal A\Lambda_{\mathbbm{1},\mathbbm{1}} $ the \textit{weight $ 0 $ centralizer algebra}, primarily for historical reasons in connection with planar algebras.  It turns out that the fusion algebra of $ \mcal C $,  $ \t{Fus}(\mcal C) $, is $ * $-isomorphic to $ \mcal A\Lambda_{\mathbbm{1},\mathbbm{1}} $ (See \cite[Proposition 3.1]{GJ}).

The \textit{annular category with weight set $ \Lambda $} is the category with objects space as $ \Lambda $ and the morphism space from $ b $ to $ c $ as $ \mcal A\Lambda_{b,c} $.
Composition is given by the multiplication defined above.
Both the algebra as well as category are often denoted by $ \mcal A\Lambda $.
Since both of these essentially contain the same information, they are used interchangeably.

The \textit{tube algebra}, $ \mcal A \mcal C $ is (by a slight abuse of notation) the annular algebra the weight set $ \Irr(\mcal C) $.
This algebra was first introduced by Ocneanu (\cite{O}).
A weight set $ \Lambda \subseteq $ is said to be \textit{full} if every simple object is equivalent to subobject of some  $b\in \Lambda $.
By \cite [Proposition 3.5]{GJ}, any annular algebra with full weight set is strongly Morita equivalent to the tube algebra.

The representation category $ Rep(\mcal A\Lambda) $ is the category of non-degenerate $ * $-representations of $ \mcal A\Lambda $ as bounded operators on a Hilbert space, with bounded intertwiners as morphisms.
This is a W*-category.
By our above comments, whenever $ \Lambda $ is full, we have $ Rep(\mcal A\Lambda) \cong Rep(\mcal A)$ as W*-categories, and thus it makes sense to talk about \textit{the category} of annular representations, which can be realized as the representation category of any annular algebra with full weight set.
We shall see in Section \ref{ann rep free prod} that the weight set can further be reduced in some cases without affecting the resulting category of annular representations.  

One of the reasons these categories are nice is that the tube algebra (or any full annular algebra) admits a universal C*-algebra, $C^{*}(\mcal{A}\Lambda)$, such that $ Rep(\mcal A\Lambda)\cong Rep(C^{*}(\mcal{A}\Lambda))$, where the latter is the category of non-degenerate, continuous $ * $-homomorphisms from the C*-algebra $C^{*}(\mcal{A}\Lambda)$ to $\mcal B(\mcal H)$.
For example, this tells us that the category decomposes as a direct integral of factor representations.
  
One way to access this category is to understand the representation theory of the unital centralizer algebras $\mcal{A}\Lambda_{a,a}$.
If $ a \in \Lambda $, a linear functional $ \phi:\mcal A\Lambda_{a,a}\rightarrow \C $ with $\phi(1)=1$ is said to be \textit{weight a annular state}, or an \textit{admissible state}, if $ \phi(f^\#\cdot f)\geq 0 $ for every $ f\in \mcal{A}\Lambda_{a,b} $ and $b \in \Lambda  $.

Using a GNS construction, each annular state gives a non-degenerate representation of $ \mcal A \Lambda$ (see \cite [Section 4]{GJ}).
Annular states provide a useful way of constructing representations of whole algebra by looking at representations of much smaller centralizer algebras or even subalgebras of the tube algebra.
A representation $ (\pi,\mcal H) $ of a centralizer algebra $ \mcal{A}\Lambda_{a,a} $ is said to be \textit{admissible} if there exists a representation $( \wt\pi,\wt {\mcal H} )$ of $ \mcal{A}\Lambda $ such that $ (\wt\pi, \wt {\mcal H})\big|_{\mcal{A}\Lambda_{a,a}} $ is unitarily equivalent to $ (\pi,\mcal H) $.
There are several equivalent conditions for a representation of centralizer algebra to be admissible.
One such condition is that every vector state in $ (\pi,\mcal H) $ is an annular (i.e. admissible) state.
It turns out that we can construct a universal C*-algebra $ C^*_u(\mcal A_{a,a}) $ with respect to admissible representations, so that admissibility of $ (\pi,\mcal H) $ is equivalent to saying that $ (\pi,\mcal H) $ extends to a representation of $ C^*_u(\mcal A_{a,a}) $.
This algebra is a corner of the universal C*-algebra of the entire tube algebra, so all the pieces fit together nicely.

\section{Annular algebra of free product of categories} \label{ann rep free prod}

We will characterize the annular algebra of $ \mcal C* \mcal D $ where $ \mcal C $ and $ \mcal D $ are rigid, semi-simple C*-tensor categories with simple unit objects.  We note that while providing definitions of the free product $\mcal{C}_{\pm}$ was more convenient to distinguish the two categories, while in this section, using $\mcal{C}$ and $\mcal{D}$ seems better.
By \cite{GJ}, the annular representation category can be obtained from representations of any annular algebra with respect to any full weight set in Obj($\mcal C* \mcal D $) (in particular, a set of representatives of the isomorphism classes of simple objects).

However, in our case, we can actually work with a smaller, non-full weight set, and still capture the entire category. 
To describe this weight set, let $ \textbf{I}_{\mcal C} $ (respectively $ \textbf{I}_{\mcal D} $) be a set of representatives of the isomorphism classes of simple objects in $ \mcal C $ (respectively $ \mcal D $) \textit{excluding the isomorphism class of the unit object}.  
Recall that the set of words (including the empty one) with letters coming alternatively from $ \textbf{I}_{\mcal C}  $ and $ \textbf{I}_{\mcal D} $ is in bijective correspondence with the set of isomorphism classes of simple objects $\Irr(\mcal{C} * \mcal{D})$, where the empty word corresponds to the tensor unit in $\mcal{C}*\mathcal{D}$.
We define $\textbf{W}$ to be the subset of these words with \textit{strictly positive and even} length, such that the first letter comes from $\textbf{I}_{\mcal C}$.
We will say a positive length word is a $\mathcal{C}\t{-}\mathcal{D}$ word if it starts with a letter of $\mcal C$ and ends with a letter of $\mcal D$, and extend this terminology in the obvious way.
%Elements of $ \mcal I_c $, $ \mcal I_d $ and words with letters in $ \mcal I_c \sqcup \mcal I_d $ (in particular, $ \Gamma$) can be identified with the corresponding simple objects in $\mcal E$ \red{as described in the prelim}.
We define the weight set $ \Lambda := \{\emptyset\} \cup \textbf{I}_{\mcal C}  \cup \textbf{I}_{\mcal D}  \cup \W$, which we note is not full.  Indeed, the alternating words of odd length and the alternating words of even length starting with a letter from $\ID$ do not appear in $\Lambda$.  Nevertheless, we have the following result:
%With a slight deviation from our earlier notation, in this section, we use lower italics (\textit{viz.,} $ v,w $ etc.) to represent elements of $ \wt\Lambda $.
%By an empty word $ \emptyset $ we mean the tensor unit $ \mathds{1} $.
%It should be understood from the context whether it represents elements of $ W $(or $W_1 $ or $ W_2$) or $ \wt\Lambda $.

\begin{lem}\label{red wt}
$Rep(\mcal A \Lambda)$ and the representation category of the tube algebra $\mcal A$ of $ \mcal C \ast \mcal D $, are unitarily equivalent as linear $ * $-categories.
\end{lem}
\begin{proof}
Clearly, the restriction functor $Res: Rep(\mcal A)\longrightarrow Rep(\mcal A \Lambda)$ is a linear $*$-functor.
We begin by showing that $Res$ is essentially surjective.

Given a representation $ (\pi,V) $ of $\mcal A \Lambda $ and $w\in \Irr(\mcal{C}*\mcal{D})$, we consider the vector space $\bigoplus\limits_{v \in \Lambda} \{{\mcal A}_{v,w} \bigotimes V_v\}$.
We define a sesquilinear form $\langle \cdot, \cdot \rangle$ on this vector space by $\langle y_1\otimes \xi_1,y_2\otimes \xi_2 \rangle_w := \langle \pi(y_2^{\#}\cdot y_1)\xi_1,\xi_2 \rangle _{v_2}$, where $y_{i}\in {\mcal A}_{v_{i},w}$ and $\xi_{i}\in V_{v_i}$.

We first want to show that $\langle x , x\rangle_w \geq 0$ for any vector $ x= \sum \limits^n_{i=1} y_i \otimes \xi_i$.
But we have $\langle x,x\rangle_w = \left\langle T \xi , \xi \right\rangle$, where $T =  \left(\pi(y^{\#}_i\cdot y_j)\right)_{i,j}: \bigoplus \limits^n_{i=1} V_{v_i} \ra \bigoplus \limits^n_{i=1} V_{v_i}$, and $\xi = (\xi_i)_i\in  \bigoplus \limits^n_{i=1} V_{v_i} $.

If $ w \in \Lambda $, then $T$ is clearly a positive operator and hence we have non-negativity of $\langle x , x\rangle_w$.
Suppose now that $ w $ has even length and its first letter is in $ \ID $, say $ w = d_1 c_1 d_2 c_2 \ldots d_k c_k $.
Consider the word $ w' = c_1 d_2 c_2 \ldots d_k c_k d_1  \in \Lambda$.  Let $ \rho \in \mcal A_{w^{\prime},w} $ be the canonical rotation unitary.
Then, for any $y\in \mcal A_{v,w} $, there is a unique $y^{\prime}\in \mcal A_{v,w^{\prime}}$ such that $y=\rho\cdot y^{\prime}$.
Thus we have
$$T =  \left(\pi(y^{\#}_{i}\cdot y_j)\right)_{i,j}=  \left(\pi\left((\rho\cdot y^{\prime}_i)^{\#}\cdot (\rho\cdot y^{\prime}_j)\right)\right)_{i,j}=\left(\pi(y^{\prime \#}_i \cdot y^{\prime}_j)\right)_{i,j},$$
hence positivity follows from the previous case.
Defining $\overline{\Lambda}$ to be the union of $ \Lambda $ and the set of words of even length (regardless of starting character), we have just shown positivity for weights in $\overline{\Lambda}$.

Now suppose $ w$ has odd length; say $ w= a_{-k} \ldots a_{-1} a_0 a_1 \ldots a_k$.  
Note that the $ a_{2l} $'s are either all in $\IC$ or all in $ \ID$, and similarly for the odd letters.
Now define the word $w' = a_0 a_1 \ldots a_k a_{-k} \ldots a_{-1}$.
This word no longer represents an isomorphism class of simple object, however the object it represents is isomorphic to a direct sum of simple objects, all of which have even length, i.e., $w^{\prime}\cong \oplus_{s} u_{s}$, where $u_{s}\in \overline{\Lambda}$.
Let $p_{s}\in \left(\mcal{C}*\mcal{D}\right)(u_{s}, w^{\prime})$ be isometries such that $\sum \limits_{s} p^{\ }_{s} p^{*}_{s}=1_{w^{\prime}}$ (which automatically implies $ p^\ast_sp^{\ }_t=\delta_{s,t}\;1_{u_s} $).

Let $\mcal A \textbf{Obj}$ denote the annular algebra whose weight set consists of all isomorphism classes of objects in $\mcal{C}*\mcal{D}$, and pick any rotation $\rho \in \mcal A \textbf{Obj}_{w^{\prime},w}$ (which is automatically unitary).
Then any element $y_{i}\in \mcal A_{v_i,w}$ can be written $y_i=\sum \limits_{s} \rho \cdot p_{s}\cdot y^{\prime}_{s,i}$, where $y^{\prime}_{s,i} \coloneqq p^*_s \cdot \rho^\# \cdot y_i  \in \mcal A_{v_{i},u_{s}}$.
% and $u_{s}$ has even length.
%%%%%%%%%%%%%%%%%%%%%%%%%%%%%%%%%%
\comments{Let $m$ denote the number of simple objects occurring in the decomposition of $w^{\prime}$ (counting multiplicities).  

Defining the operator $T^{\prime}= \left(\pi(y^{\prime \#}_{s,i}\cdot y^{\prime}_{t,j})\right)_{(s,i),(t,j)}\in M_{m}(\mcal A \overline{\Lambda})\otimes M_{n}(\mcal A \overline{\Lambda})$, which is positive in this algebra by our previous argument, we see that 

$$T =  \left(\pi(y^{\#}_i \cdot y_j)\right)_{i,j}=  \left(\pi((\sum_{s}\rho \cdot p_s \cdot y^{\prime}_{s,i})^{\#} (\sum_{t}\rho \cdot p_t \cdot y^{\prime}_{t,j})\right)_{i,j}=\left(\sum_{s}\pi(y^{\prime \#}_{s,i} \cdot y^{\prime}_{s,j})\right)_{i,j}$$
$$=Tr\otimes id(T^{\prime})$$

Since $Tr\otimes id:  M_{m}(\mcal A \overline{\Lambda})\otimes M_{n}(\mcal A \overline{\Lambda})\rightarrow M_{n}(\mcal A \overline{\Lambda}) $ is a positive map (indeed, it is a positive scalar times the canonical trace preserving conditional expectation of the inclusion $ 1\otimes M_{n}(\mcal A \overline{\Lambda})\subseteq M_{m}(\mcal A \overline{\Lambda})\otimes M_{n}(\mcal A \overline{\Lambda}) )$, we see that $T$ is positive,  
}%%%%%%%%%%%%%%%%%%%%%%%%%%%%%%%%%%
Observe that 
\begin{equation*}
\begin{split}
T & =  \left(\pi(y^{\#}_i \cdot y_j)\right)_{i,j} =  \left(\pi \left( \left[\sum_{s}\rho \cdot p_s \cdot y^{\prime}_{s,i} \right]^{\#} \cdot \left[ \sum_{t}\rho \cdot p_t \cdot y^{\prime}_{t,j} \right] \right) \right)_{i,j} \\
& = \sum_t \left( \pi \left( [y'_{t,i}]^\# \cdot y'_{t,j} \right)  \right)_{i,j}
\end{split}
\end{equation*}
which is positive as all $ u_t $'s are in $ \ol \Lambda $ and hence our argument is complete.

\medskip

Now that we have shown $\langle x, x\rangle_{w}\ge 0$, we can define $Ind(V)_{w}$ as the Hilbert space obtained by the completion of the quotient of our vector space over the null space of the inner product.
Before quotienting and completing, our vector space has the obvious action of $\mcal{A}$.  Our above argument shows that $ \langle \pi (\cdot) x , x \rangle_w$ is a positive annular functional.
Thus by \cite[Lemma 4.4]{GJ}, we have a well-defined, bounded, $*$-action of the tube algebra $\mcal{A}$ on $Ind(V)$.
It is now easy to verify that $Res\circ Ind(V)\cong V$ via the interwiner defined by sending $\sum_{i} y_{i}\otimes \xi_{i}$ to $\pi(y_i)\xi_{i}$.

Now to prove that $Res$ is fully faithful, we first claim that any representation $ (\theta ,\mcal H) \in \t {Rep} (\mcal A)$ is generated by $ \us{w \in \Lambda} \bigcup \mcal H_w $.
We need to check $$\mcal H^0_w := \t{span} \left\{ \theta (x) \xi : x\in \mcal A_{v,w} ,\, \xi \in \mcal H_{v} ,\, v \in \Lambda  \right\} $$ is dense in $\mcal H_w $ for all $w\in \Irr(\mcal{C}*\mcal{D}) \setminus \Lambda $; we will, in fact, show $\mcal H^0_w =\mcal H_w $.
Now, $ w \in \Irr(\mcal{C}*\mcal{D}) \setminus \Lambda  $ implies $ \abs w \geq 2 $.
Suppose $ w $ is of $ \mathcal{D} $-$ \mathcal{C}$ type, so that $ w = d u $ for some $ u $ of $ \mcal{C} $-$ \mcal{C} $ type. We have the unitary rotation 

$$\rho := 1_d \otimes 1_{u} \otimes 1_d \in \left(\mcal{C}*\mcal{D}\right)(d w^{\prime} , w d) = \mcal A^d_{w^{\prime}, w} \subset \mcal A_{w^{\prime}, w},$$ where $ w^{\prime} = u d \in \Lambda $, whose $ \theta $-action implements a unitary from $\mcal H_{w^{\prime}} $ to $\mcal H_w $; so, $\mcal H^0_w =\mcal H_w $.

The remaining elements of $\Irr(\mcal{C}*\mcal{D}) \setminus \Lambda $ are words of types $ \mathcal{C} $-$ \mathcal{C}$ or $ \mathcal{D} $-$ \mathcal{D} $ type, which neccessarily have odd length $\ge 3$.
Consider such a $ w $, say $ w= a_{-k} \ldots a_{-1} a_0 a_1 \ldots a_k$.
As above, the even $ a_{i} $'s are either all in $\IC $ or all in $ \ID$.
Let $w' : = a_0 a_1 \ldots a_k \otimes a_{-k} \ldots a_{-1} $ or $a_1 \ldots a_k \otimes a_{-k} \ldots a_{-1} a_0$ depending on whether $ a_0  \in \IC$ or $ \ID $, and $ \rho' $ be the rotation unitary from $ w $ to $ w'$.
Note that $ w' $ may no longer be simple; however, it decomposes into a direct sum of simple objects all of which either have even length or lie in $ \Lambda $ (using the fusion rule).
Suppose $ w' \cong \us i \oplus u_i$ is the simple object decomposition.
%Note that $w^{\prime}$ no longer represents an isomorphism class of simple object; however the object it represents is isomorphic to a direct sum of simple objects, all of which have even length and lie in $\Lambda$ (by definition of $w^{\prime}$), i.e., $w^{\prime}\cong \oplus_{i} u_{i}$, where $u_{i}\in \Lambda$.
Let $p_{i}\in \left(\mcal{C}*\mcal{D}\right)(u_{i}, w^{\prime})$ be isometries such that $\sum_{s} p^{\ }_{i} p^{*}_{i}=1_{w^{\prime}}$.
%(implying $p^{*}_{j}p^{\ }_{i}=\delta_{i,j} 1_{u_{i}}$).
Set $ x_i := (\rho')^\# \cdot p_i  \in \mcal A_{u_i,w}$.
Clearly, $ \sum_{i} x_i \cdot x^\#_i = 1_w $ (in $ \mcal A_{w,w} $).
Since the $ u_i$'s belong to $ \Lambda $, any $ \xi \in \mcal H_w $ can be expressed as $ \sum_{i} \theta (x_i) [\theta (x^\#_i) \xi] \in \mcal H^0_w $.

\vskip 1em
Thus our claim that any representation is generated by the $\Lambda$ weight spaces is proven.
This immediately implies that the restriction functor is faithful.
It also shows that $Res$ is full.  Indeed, consider a morphism $ f: Res(\pi ,\mcal H) \ra Res(\gamma , \mcal K)$ in $ Rep(\mcal A \Lambda) $.
For $ w \in \Irr(\mcal{C}*\mcal{D}) \setminus \Lambda $, if an $\mcal A$-linear extension of $f$ exists we see that $f(\sum \pi(y_{i})\xi_{i})=\sum \gamma(y_i)f(\xi_i)$, for $y_{i}\in \mcal{A}_{v,w}$, $v\in \Lambda$, and $\xi\in \mcal H_{v}$.
Indeed, this will serve as a definition of the extension, but we must show it is well defined.
Suppose $\sum \limits_{i} \pi(y_{i})\xi_{i}=0$.
Then for any fixed $j$, $\sum \limits_{i} \pi(y^{\#}_{j}\cdot y_{i})\xi_{i}=0$.
Since $y^{\#}_{j}\cdot y_{i}\in \mcal{A} \Lambda$, we have

$$\sum_{i,j}\langle \gamma(y_{i})f(\xi_i), \gamma(y_{j})f(\xi_j)\rangle_{\mcal K}=\langle \gamma(y^{\#}_{j}\cdot y_{i})f(\xi_i), f(\xi_j)\rangle_{\mcal K}$$

$$=\sum_{j} \sum_{i}\langle \pi(y^{\#}_{j}\cdot y_{i})\xi_i, f^{*}f(\xi_j)\rangle_{\mcal H}=0$$
It is easy to see that the extension of $f$ remains bounded.  This concludes the proof.
\comments{
we extend $ f $ in the following way
\[
H_w \ni \xi \os {\displaystyle0\widetilde{S}}\longmapsto
\left\{\begin{tabular}{rl}
$ \vphi (\rho)\; S\; \theta (\rho^\#) $, & if $ w $ is of $ d $-$ c $ type,\\
$\sum \limits^n_{i=1}  \vphi(x_i)\; S\; \theta (x^\#_i) $, & if $ w $ is of $ c $-$ c $ or $ d $-$ d $ type,
 \end{tabular}\right.
\]
where we use the notations $ \rho $ and $ x_i $'s defined in the proof of Assertion 1.
It is routine to verify that $ \widetilde T $ is indeed $ \mcal A $-linear. \red{(typical case by case checking, should we elaborate?)}
}
\end{proof}

We proceed to the study of the $ * $-algebra $ \mcal A\Lambda $.
We divide this into subsections corresponding to the length (denoted by $ \abs \cdot $) of the words in $ \Lambda $.
Since the empty word (that is, zero length word) stands for the tensor unit of $ \mcal{C}*\mcal{D} $, the centralizer algebra $ \mcal A \Lambda_{\emptyset,\emptyset} $ is isomorphic to the fusion algebra, and we will be able to describe admissible representations of these in terms of representations of free product C*-algebras.  Thus in this section, we will focus on the structure of $\mcal A \Lambda_{v,w}$ for words $v,w\in \Lambda$ of positive length.
By $ \mcal A\mcal C $ (resp., $ \mcal A\mcal D $) we mean the tube algebra/category of $ \mcal C $ (resp. $ \mcal D $).
%If $ w \in \t{Irr}(\mcal E) $, then $ w=w_1w_2\ldots w_n $ with $ w_i $'s coming alternatively from $ \t{Irr}(\mcal C) \t{ and Irr}(\mcal D) $.
%We say $ w $ is of $ c\t{-}c $ (resp. $ c\t{-}d$, $ d\t{-}d$, $d\t{-}c $) type if $ w_1 \in \t{Irr}(\mcal C) \t{ and } w_n \in \t{Irr}(\mcal C)  $(resp. $ w_1 \in \t{Irr}(\mcal C) \t{ and } w_n \in \t{Irr}(\mcal D) $, $ w_1 \in \t{Irr}(\mcal D) \t{ and } w_n \in \t{Irr}(\mcal C) $, $ w_1 \in \t{Irr}(\mcal D) \t{ and } w_n \in \t{Irr}(\mcal D)   $)  
\subsection{Words of length at least $ 2 $}

Define a relation on $\W$ by $w_{1}\sim w_2 $ if and only if $ w_1=uv $ and $w_2=vu$ for some subwords $u,v$.
Clearly, $ \sim $ defines an equivalence relation on $\W$.  Obviously if $ w_1 \sim w_2 $, then $ |w_1| = |w_2| $.
\begin{lem} \label{equiv}
	For $w_1, w_2 \in \W $, $ \mcal A\Lambda_{w_1.w_2} \neq \{0\} $ if and only if $ w_1\sim w_2 $.	
\end{lem}
\begin{proof}
Suppose $ w_1\sim w_2$ so that  $ w_1=uv $ and $w_2=vu$.
Consider the rotation $ \rho:= (1_{v}\otimes \bar R_{u})(R^{*}_{u}\otimes 1_{v}) \in \left(\mcal{C}*\mcal{D}\right)(\bar{u}w_1,w_2\bar{u})\subseteq \mcal{A}_{w_1, w_2} $ for any standard solution $ (R_{u} , \ol R_{u}) $ to the conjugate equation for $ (u , \ol u) $.
It is non-zero (since it is unitary) and hence $ \mcal A\Lambda_{w_1.w_2} \neq \{0\} $.

Now suppose $ \mcal A\Lambda_{w_1.w_2} \neq \{0\}$ and without loss of generality, let  $ w_1 \neq w_2 $.
Then there exists $ v \in \Irr(\mcal{C}*\mcal{D})$ (of length, say, $ m>0 $) such that $ \mcal A\Lambda_{w_1.w_2}^v \neq \{0\} $.
%$ v $ being simple in $ \mcal E $, it is a word with alternating letters from $ \t{Irr}(\mcal C) \t{ and Irr}(\mcal D) $.
%Let $ v = v_1v_2\ldots v_m $ with $ v_i $'s belonging alternatively to $ \mscr I_c $ and $ \mscr I_d $. 
%	We will have four cases depending on the type of $ w_1 $ and $ w_2 $.
%	We shall consider only one case here and proof in other cases is by similar arguments.	
%	Let $ w_1 $ and $ w_2 $ are both of $ c\t{-}d $ type.		
Suppose $ m $ is odd. Then $ v $ is either of $ \mcal{C}$-$\mcal{C}$ type or $ \mcal{D}$-$\mcal{D}$ type.
If $ v $ is of $ \mcal{C}$-$\mcal{C}$ type (resp. $ \mcal{D}$-$\mcal{D}$ type), then $ w_2v $ (resp. $ vw_1 $) is simple and is of odd length, whereas $ vw_1 $ (resp. $ w_2v $) is not simple and any simple subobject will be of length strictly smaller than that of $ vw_1 $.
Hence $ \mcal A\Lambda_{w_1.w_2}^v = \{0\} $ which is a contradiction.
So $ m $ cannot be odd.
	
Thus $m$ must be even, so $v $ can be of  $ \mcal{C}$-$\mcal{D}$ or $ \mcal{D}$-$\mcal{C}$ type.
It is enough to consider the case where $ v  $ is of $ \mcal{C}$-$\mcal{D}$ type, since the other case will follow by taking $ \# $.
As $ w_1, w_2\in \W $, $ vw_1$ and $ w_2v $ are simple.
Therefore, $ \mcal A\Lambda_{w_1.w_2}^v \neq \{0\} $ implies the equality
\begin{equation}\label{gcd}
vw_1=w_2v
\end{equation}

In particular, we see that $ w_1 $ and $ w_2 $ have the same length, say $ n $.

If $ m = n $, then Equation \ref{gcd} implies $ w_1 = v = w_2 $ which is not possible by assumption.
Suppose $ m < n $.
By Equation \ref{gcd}, there exists a word $ u $ such that $ w_2 = vu $.
So, $ vw_1 = vuv $ implying $ w_1 = uv $, and thus $ w_1 \sim w_2 $.

We are left with the case when $ m>n $.
Equation \ref{gcd} tells us that $ v $ starts with the subword $ w_2 $; say $ v=w_2 v' $.
Plugging this into Equation \ref{gcd}, we get $ v' w_1 = w_2 v' $.
Note that $ \abs{v'} =n-m$.
If length of $ v' $ is not less than or equal to $ n $, then we repeat the above argument with $v^{\prime}$.
Since $\abs{v^{\prime}}<\abs{v}$, we will eventually find some tail-end subword of $v$, say $v_{0}$, such that $v_{0}w_1=w_2 v_{0}$ with $\abs{v_{0}}\le n$.
Then we apply the previous cases.
%, which in turn is possible only if $ w_1=vw' \t{ and } w_2=w'v \t{ for some }w' \in \Lambda$.
%Hence $w_1 \sim w_2 $.
\end{proof}

Using similar techniques, we also have the following lemma:

\begin{lem}\label{uneq len}
Let $w\in \W$.  For any $v\in \Lambda\setminus \W$, $\mcal A_{v,w}=\{0\}$.
\end{lem}

\begin{proof}
First we consider the case $v=\emptyset$.
In general, $\mcal A_{\emptyset,w}\ne \{0\}$ implies that $w$ is an object in the adjoint sub-category of $\mathcal{C}*\mathcal{D}$, or in other words, $w$ is isomorphic to a sub-object of $u\bar{u}$ for some simple object $u\in \mathcal{C}*\mathcal{D}$.
If $u$ is length $0$, then obviously $\abs w=0$, a contradiction.
If $u$ has length greater than or equal to $1$, as every word that appears as a sub-object of $v\bar{v}$ is of $\mcal{C}$-$\mcal{C}$ or $\mcal{D}$-$\mcal{D}$ type, $ w $ cannot be a sub-object of $ u\bar{u} $, which implies that $\mcal A_{\emptyset,w}=\{0\}$.

Now we consider the case that $v$ has length $1$.
First assume $v\in \mcal C$.
If $\mcal A_{v,w}\ne\{0\}$, then there is some word $u$ so that $\left(\mcal{C}*\mcal{D}\right)(uv,wu)\ne\{0\}$, which is equivalent to $\left(\mcal{C}*\mcal{D}\right)(v\bar{u},\bar{u}w)\ne \{0\}$.
First suppose $\abs{u}$ is odd.
If it is of $\mcal{C}$-$\mcal{C}$ type, then $wu$ is simple, and $uv$ is isomorphic to a direct sum of simple objects each of which have length strictly smaller than the length of $wu$, so the morphism space must be $0$.
Similarly if $u$ is of $\mcal D$-$\mcal D$ type, then so is $\bar{u}$, and our hypothesis implies $\left(\mcal{C}*\mcal{D}\right)(v\bar{u},\bar{u}w)\ne \{0\}$.
In this case, both words are simple, but $\abs{v\bar{u}}<\abs{\bar{u}w}$, and thus the morphism space must be $\{0\}$.

Thus we are left to consider the case when $\abs{u}$ is even.
If $u$ is $\mcal{C}$-$\mcal{D}$ type, then $wu$ is simple, and the length is strictly greater than the length of any subobject of $uv$ (since $\abs{v}=1$) a contradiction.
If $u$ is $\mcal{D}$-$\mcal{C}$ type, then $\bar{u}w$ is simple with length strictly greater than the length of any simple sub-object of $v\bar{u}$.  

The case with $v\in \mathcal{D}$ is entirely analogous.
\end{proof}

\comments{
\begin{rem}\label{word equiv implies rep equiv}
\sout{Suppose $(\pi,\mcal H)$ is a representation of  $ \mcal A\Lambda$ and $w_1 \sim w_2$ in $ \Gamma $.
Let $ \rho $ be as in proof of \ref{equiv}. Then $ u $ is a unitary and $ \mcal A\Lambda_{w_1,w_1}=u^*\mcal A\Lambda_{w_2,w_2}u  $. 
Hence ($ \pi_{w_1},\mcal H_{w_1} $) and ($ \pi_{w_1},\mcal H_{w_1} $) are conjugate to each other (where ($ \pi_{w},\mcal H_{w} $) is the restriction of $(\pi,\mcal H)$ to $ \mcal A\Lambda_{w,w}$).}
\end{rem}
}

\begin{lem}\label{wt > 1}
	For $ w \in \W $, the centralizer algebra $ \mcal A\Lambda_{w,w}$ is isomorphic to the group algebra $\C[\Z]$ as $*$-algebras.
	%Thus, irreducible representations of $ \mcal A \Lambda $  which have support at lowest weight $ 2 $ can be characterized uniquely by $ S^1\times\mcal S $.
\end{lem}
\begin{proof}
Let $v$ be a subword of $w$ such that $ w=v^k=\underbrace{vv\ldots v}_{k\t {-times}} $, for largest possible positive integer $ k $.
We will say that $ w $ is \textit{maximally periodic with respect to} $ v $. 
Note that $ v $ must be of $ \mcal C$-$ \mcal D$ type.
Consider the (unitary) rotation 

$$ \rho^v_{w,w} := 1_{v^{k+1}} \in \left(\mcal{C}*\mcal{D}\right)(vw,wv) = \mcal A \Lambda^v_{w,w}$$
 whose inverse is given by 

 $$ \left(\rho^{v}_{w,w}\right)^\# = (1_{v^{k-1}} \otimes \ol R_v) (R^*_v \otimes 1_{v^{k-1}}) \in \left(\mcal{C}*\mcal{D}\right)(\ol v w, w \ol v) = \mcal A \Lambda^{\ol v}_{w,w}$$ for any standard solution $ (R_v, \ol R_v) $ of the conjugate equation for $ (v , \ol v) $.

Note that for any $ n\in \Z $, $ \left(\rho^{v}_{w,w}\right)^n \in \mcal A \Lambda^{v^n}_{w,w} $ with the convention $v^{-1}=\bar{v} \t{ and } v^0 := \mathbbm{1}$.
Thus, $ \left\{ \left(\rho^{v}_{w,w}\right)^n : n \in \Z \right\} $ is an orthogonal sequence in $ \mcal A \Lambda_{w,w} $ with respect to the canonical trace.
Hence, we have an injective homomorphism from $\C [\Z] $ to $ \mcal A \Lambda $ sending the generator of $\Z$, which we denote $g$, to $ \rho^{v}_{w,w} $.
It remains to show that the homomorphism is surjective.

We now claim that if $ u\in \Irr(\mcal{C}*\mcal{D})$, then $ \mcal A\Lambda^u_{w,w}= \left(\mcal{C}*\mcal{D}\right)(uw,wu) \neq \{0\} $ if and only if $ u=v^n \t{ for some } n \in \Z$.

By the same argument as in proof of ``if" part of Lemma \ref{equiv}, it is easy to deduce that $ u $ must be any one of $ \mcal{C} $-$ \mcal{D} $ or $ \mcal{D}$-$\mcal{C}$ types if $ \mcal A\Lambda^u_{w,w}= \left(\mcal{C}*\mcal{D}\right)(uw,wu) \neq \{0\}$.
It suffices to consider the case of $ \mcal{C}$-$\mcal{D}$ type $ u $, since the other case will follow from this by applying $ \# $.

Since both $ u$  and   $w$ are of $ \mcal{C}$-$\mcal{D}$ type, both $ uw $ and $ wu $ are simple, $ \left(\mcal{C}*\mcal{D}\right)(uw,wu) \neq \{0\}$ implies $ uw=wu $.
Now, consider the bi-infinite word $ \ldots u w u w u w \ldots $.
If $ m = \abs u $ and $ n = \abs w $, then by the commutation of $ u $ and $ w $, we may conclude that the infinite word  is both $ m $- and $ n $-periodic, and thereby, $ l := \gcd (m,n) $-periodic.
So, there exists a word $ v^{\prime} $ of length $ l $ such that both $ u $ and $ w $ are integral powers of $ v^{\prime} $.
Since $ w $ is maximally periodic with respect to $ v $, $\abs v \leq \abs{v^{\prime}} $, which will then imply that $ v^{\prime} $ is an integral power of $ v $.
Hence, $ u $ is an integral power of $ v $.

We will be done if we can show $ \mcal A \Lambda^{v^n}_{w,w}  = \C  \rho^{v^n}_{w,w} $ for $ n\in \Z $.
Again, it is enough to show for $ n\geq 0 $ since the other cases follow by taking $ \# $.
If $ n \geq 0 $, however, then $ \mcal A \Lambda^{v^n}_{w,w} = \left(\mcal{C}*\mcal{D}\right)(v^{k+n} , v^{k+n}) $ is one-dimensional (by simplicity of $ v^{k+n} $).
\end{proof}

Via the inclusion $ \W \subset \Lambda $, we may consider $ \mcal A \W $ as a $ * $-subalgebra of $ \mcal A \Lambda $.
In fact, by Lemma \ref{uneq len}, we see that $A\W$ is actually a summand of $A\Lambda$.
The above lemma now allows us to identify $A\W$.
Let $\W_{0}=\W/\sim $, the set of equivalence classes of words in $\W$ modulo the cyclic relation $\sim$ defined in the beginning of this section.
Recall that $M_{n}(\C)$ denotes the algebra of $n\times n$ matrices.

\begin{cor}\label{wt>1corner}
$\mcal{A}\W$ is a direct summand of the algebra $\mcal{A}\Lambda$.  Moreover, as $ * $-algebras 

$$ \mcal A \W \cong \bigoplus_{[w]\in \W_{0}}  M_{|w|}(\C)\otimes \C[\Z].$$ 
\end{cor}
\begin{proof}
As explained above, the first statement follows from Lemma \ref{uneq len}.

For the second one, we pick a representative $ w\in [w]\in \W_0 $.
Then for any other $v\in [w]$, it is clear from Lemma \ref{wt > 1} that $\mcal{A}\W_{w,v}\cong \C[\Z]$ as a vector space, where $\Z$ is identified with powers of unitary rotation operators $ \sigma_v \in \mcal A\Lambda_{w, v} $ for all $ v \in [w] $.
Note that $\mcal{A}\W_{w,v}=\{0\}$ for $v\notin[w]$ by Lemma \ref{equiv}.

The required isomorphism is given by the map defined for $w_1,w_2\in [w]$ and $x\in \mcal A S_{w_1,w_2}$ by
\[
  x \longmapsto E_{w_1,w_2} \otimes \sigma_{w_2} \;x \;\sigma^\#_{w_1} \in M_{|w|}(\C) \otimes \mcal A \Lambda_{w,w} \cong  M_{|w|}(\C)\otimes \C[\Z].
\]
\end{proof}

\subsection{Words of length $ 1 $}

% with a chosen set $ \mcal I $ of representatives of equivalence classes of simple objects.
For a rigid C*-tensor category $\mcal{C}$, we let $\textbf{S}(\mcal{C}):=\{[a]\in\Irr(\mcal{C})\ :\ N^{a}_{b\overline{b}}\ne 0\ \text{for some}\ [b]\in \Irr(\mcal{C})\}$.
$\textbf{S}(\mcal{C})$ tensor generates the \textit{adjoint subcategory} of $\mcal{C}$, which is the trivial graded component with respect to the universal grading group, but in general $\textbf{S}(\mcal{C})$ gives a proper subset of the simple objects in the adjoint subcategory.

%We say $ w $ is of $ T $\textit{-type w.r.t to} $ v $, if $ v $ is a non-trivial simple object such that $ v\bar{v} $ contains $ w $ as a subobject, i.e., Hom($ v,wv $)$ \neq \{0\} $.
\comments{Let $ w \in \Lambda$ of length 1.
So, $ w \in [\t{Irr(}\mcal C\t{)}]\backslash\{1\} $ or $ w \in [\t{Irr(}\mcal D\t{)}]\backslash\{1\} $.
It is said to be of $ T$\textit{-type} w.r.t $ v $, if there exists $ v $ in $ [\t{Irr(}\mcal C\t{)}]\backslash\{1\} $ if $ w \in [\t{Irr(}\mcal C\t{)}]\backslash\{1\} $, and in $ [\t{Irr(}\mcal D\t{)}]\backslash\{1\} $ if $ [\t{Irr(}\mcal D\t{)}]\backslash\{1\} $ such that $ w $ is a sub-object of $ v\tilde{v} $ i.e., Hom$ (v,wv) \neq \{0\}$.
}

\begin{lem}\label{adjointweight0}
Let $w\in \IC$.  Then $ \mcal  A\Lambda_{\emptyset,w} \neq \{0\}$ if and only if $ w $ belongs to $ \textbf{S}(\mcal{C}) $.
The same holds replacing $\mcal{C}$ with $\mcal{D}$.
\end{lem}
\begin{proof}
Suppose $ w \in \textbf{S}(\mcal C )$. then there is a simple $ v $ such that $\{0\} \neq  \left(\mcal{C}*\mcal{D}\right)(v,wv) = \mcal A\Lambda_{\emptyset,w}^v $ implying, $ \mcal A\Lambda_{\emptyset,w} \neq \{0\}$.
	
Now suppose $ \mcal A\Lambda_{\emptyset,w} \neq \{0\}$. 
Choose $ v \in  \Irr(\mcal{C}*\mcal{D})\setminus\{ \mathbbm{1}\} $ such that $  \mcal A\Lambda_{\emptyset,w}^{v} =\left(\mcal{C}*\mcal{D}\right)(v,wv) \neq \{0\} $.
By arguments as in the proof of  Lemma \ref{equiv}, one can see that $ v $ must be of $\mcal C$-$\mcal C$ or $ \mcal{C}$-$\mcal{D}$ type for the morphism space to be non-zero.
Let $ v= cv' $ with $ c\in \IC$.
If $ v' = \mathbbm{1} $, then we are done.
Suppose $ \abs{v'} \geq 1 $; so, $ v' $ starts in $ \ID $.
Consider the simple objects $ \{u_i: i=0,1, \ldots n\} \subset \Irr(\mcal{C}*\mcal{D}) $ that appear as subobjects in the decomposition of $ v'\bar{v'} $, with $ u_0 = \mathbbm 1 $.
Note that, for $ i\geq 1 $,  $ u_i $ is non-trivial and is of $ \mcal{D} $-$ \mcal{D} $ type (since $ v' $ is simple).
Thus, for all $ i\geq 1 $, $ cu_i \bar{c}$ is simple and is of length greater than $ 1 $, implying $ \left(\mcal{C}*\mcal{D}\right)(w,cu_i \bar{c}) =\{0\} $.
Since $ \{0\}\neq \left(\mcal{C}*\mcal{D}\right)(v,wv) \cong \left(\mcal{C}*\mcal{D}\right)(c\,v'\,\bar{v'}\, \bar{c} ,w) $, we must have $ \mcal{C}(c\bar{c} , w)=\left(\mcal{C}*\mcal{D}\right)(c\bar{c} , w) \neq \{0\} $.
So $ w \in \textbf{S}(C)$.
\end{proof}

For the statement of the next lemma, for $c\in \IC$, note that since $\mcal{C}$ is a full subcategory of $\mcal{C}*\mcal{D}$, we can view $\mcal{AC}_{c, \mathbbm{1}}\subseteq \mcal{A}\Lambda_{c, \emptyset}$.  Similarly for $d\in \ID$.

\begin{lem}\label{adjointcd}

If $c\in \IC \subseteq \Lambda$ and $d\in \ID\subseteq \Lambda$, then $\mcal{A}\Lambda_{c,d}\ne 0$ if and only if $c\in \textbf{S}(\mcal{C})$ and $d\in \textbf{S}(\mcal{D})$.
Furthermore $\mcal{A}\Lambda_{c,d}=\mcal{A}\mcal{D}_{d,\mathbbm{1}}\cdot \mcal{A}\Lambda_{\emptyset,\emptyset}\cdot \mcal{A}\mcal{C}_{c,\mathbbm{1}}$.
\end{lem} 
\begin{proof}
If $ c \in \IC $ and $ d\in \ID $, choose $ a \in \IC$ and $b\in \ID$ such that $ c $ and $ d $ are subobjects of $\bar{a} a$ and $ b \bar{b} $ in $ \mcal C $ and $ \mcal D $ respectively.
Let $ 0 \neq y_1 \in \mcal C (ac,a)$, $ 0 \neq y_2 \in \mcal D (b,db)$.
Note that $ (y_2\otimes 1_{v_1})(1_{v_2}\otimes y_1) \in  \left(\mcal{C}*\mcal{D}\right) (bac,dba)=\mcal A\Lambda^{ba}_{c,d} \subset \mcal A\Lambda_{c,d}$ is nonzero.

Conversely, let $ \mcal A\Lambda_{c,d}\neq \{0\} $.
Then there exists a non-unit simple object $ v\in \Irr(\mcal{C}*\mcal{D}) $ such that $\left(\mcal{C}*\mcal{D}\right)(vc,dv)=\mcal A\Lambda_{c,d}^{v}  \neq \{0\} $.
If $ v $ is of $ \mcal{C}$-$\mcal{C}$ (resp. $\mcal{D}$-$\mcal{D} $) type, then $\left(\mcal{C}*\mcal{D}\right)(vc,dv) =\{0\} $ as $ dv $ (resp. $ vc $) is simple of $\mcal{D}$-$\mcal{C}$ type, and any simple subobject of $ vc $ (resp. $ dv $) in $\mcal{C}*\mcal{D}$ has length smaller than that of $ dv $ (resp. $ vc $). 
Now suppose $ v $ is of $\mcal{C}$-$\mcal{D}$ type; then, both $ vc $ and $ dv $ are simple with the same length but are of different types, hence $\left(\mcal{C}*\mcal{D}\right)(vc,dv) =\{0\} $.
Thus $ v $ can only be of $\mcal{D}$-$\mcal{C}$ type.
Also since $ v\neq \mathbbm{1} $, length of $ v $ is at least 2.

Let $ v=d^{\prime}v'c^{\prime} $, where $ d^{\prime} \in \ID, c^{\prime}\in \IC$ and $ v' \in \Irr(\mcal{C}*\mcal{D})$ is either trivial or $ \mcal{C} $-$ \mcal{D} $ type.
Consider $ \bar{v} d^{\prime}v = \bar{c^{\prime}} \, \bar{v'} \, \bar{d^{\prime}} \, d \, d^{\prime} \, {v'} \, c^{\prime}$.
If $ \bar{d^{\prime}}dd^{\prime} $ does not contain $ \mathbbm{1} $ as a subobject, then the length of every simple subobject of $ \bar{v} dv $ is strictly greater than $ 1 $, and thereby $\left(\mcal{C}*\mcal{D}\right)(vw_1,w_2v) \cong \left(\mcal{C}*\mcal{D}\right)(w_1,\bar{v}w_2v ) =\{0\} $ which is a contradiction.
Thus, $ \mathbbm{ 1} $ appears as a subobject of $\bar{d^{\prime}} dd^{\prime} $
% contains the trivial object, i.e., $ \mcal D(d,w_2d)\cong \mcal D(\mathds{1},\bar{d}w_2d)=\mcal E(\mathds{1},\bar{d}w_2d) \neq \{0\} $
and hence $ d \in \textbf{S}(\mcal{D}) $.
Similarly, by considering $ v c \bar{v}$, one may deduce that $ c \in \textbf{S}(\mcal{C})$.

For the last part, let $ v = d^{\prime} v' c^{\prime} $ be as above.
Then $ vc = d^{\prime} v' c^{\prime}  c$ and $ d v = d d^{\prime} v' c^{\prime} $.  Since $v^{\prime}$ is a word of $\mathcal{C}$-$\mathcal{D}$ type of length at least 2 whose letters are all simple, by the definition of the free product category, any morphism  $ x\in \left(\mcal{C}*\mcal{D}\right)(vc, dv)$ factorizes as $x_{1}\otimes 1_{v^{\prime}}\otimes x_{2}$, where $x_{1}\in \mcal{D}(d^{\prime}, d d^{\prime})$ and $x_{2}\in \mcal{C}(c^{\prime}, c c^{\prime})$.
The result then follows. 
\end{proof}

\begin{lem}\label{adjointcc}
Suppose $ c_1,c_2 \in \IC$.
If $ v \in \Irr(\mcal{C}*\mcal{D}) $ and $ \abs v \geq 1 $, then the space $ \mcal A \Lambda^v_{c_1,c_2}\ne \{0\}$ implies  $ v $ is of $ \mcal{C}$-$ \mcal{C}$ type.  Furthermore, we have
\begin{itemize}
\item[(i)]   If $\abs{v}=1$, then $v\in \IC$ and $\mcal A \Lambda^v_{c_1,c_2}=\mcal A \mcal{C}^{v}_{c_1, c_2}$.
\item[(ii)] If $ \abs v \geq 2$  then $\mcal A \Lambda^v_{c_1,c_2}\ne 0$ implies both $ c_1$ and $c_2 $ lie in $\textbf{S}(\mcal{C})$.  Furthermore, $ \mcal A \Lambda^v_{c_1,c_2}=\mcal A \mcal C_{\mathbbm{1},c_2} \cdot \mcal A \Lambda_{\emptyset,\emptyset} \cdot \mcal A \mcal C_{c_1, \mathbbm{1}} $.
\end{itemize}

The same statement holds, replacing $\mcal{C}$ with $\mcal{D}$.

\end{lem}
\begin{proof}
Let $ c_1,c_2 \in \IC$.  And suppose $ \mcal  A\Lambda^{v}_{c_1,c_2} \neq \{0\} $, for $\abs{v}\ge 1$.

If $ v $ is of $\mcal{C}$-$\mcal{D}$ type or $ \mcal{D}$-$\mcal{C}$ type, then $ vc_1 $ (respectively, $ c_2v $) is simple, and any simple subobject of $ c_2v $ (respectively $ vc_1 $) will have length strictly smaller than that of $ vc_1 $ (respectively $ c_2v $).
Hence $ \mcal A\Lambda^v_{c_1,c_2}= \left(\mcal{C}*\mcal{D}\right)(vc_1,c_2v)=\{0\} $.
Again, we can rule out $ v $ being $ \mcal{D} $-$ \mcal{D} $ type by comparison of the two simple objects $ v c_1 $ and $ c_2 v $, which cannot be equal since one starts with $\mcal{D}$ while the other starts with $\mcal{C}$.

For $(i)$, note that for $\abs{v}=1$ and $\mcal A \Lambda^v_{c_1,c_2}\ne \{0\}$, we must have $v\in \IC$ and in this case we see that $\mcal A \Lambda^v_{c_1,c_2}=\left(\mcal{C}*\mcal{D}\right)(vc_{1},c_{2}v)=\mcal{C}(vc_{1},c_{2}v)=\mcal{A} \mcal{C}^{v}_{c_1, c_2}$.

For $(ii)$, suppose we have $\mcal A \Lambda^v_{c_1,c_2}\ne\{0\}$, with $\abs{v}\ge 2$.
By the first part of the lemma, $v$ is of $\mcal{C}$-$\mcal{C}$ type, and hence we have $v=c^{\prime}_{1}v^{\prime} c^{\prime}_{2}$, where $v^{\prime}$ is a simple word of $\mcal{D}$-$\mcal{D}$ type of length $\ge 1$.
Thus we see that for any $x\in \left(\mcal{C}*\mcal{D}\right)(vc_{1}, c_{2}v)=\left(\mcal{C}*\mcal{D}\right)(c^{\prime}_{1}v^{\prime} c^{\prime}_{2}c_{1}, c_{2}c^{\prime}_{1}v^{\prime} c^{\prime}_{2})$, from the definition of the free product category we must have $x_{1}\in \mcal{C}(c^{\prime}_{1}, c_{2} c^{\prime}_{1})$ and $x_{2}\in \mcal{C}(c^{\prime}_{2} c_{1}, c^{\prime}_{2})$ so that $x$ factorizes as $x=x_{1}\otimes 1_{v^{\prime}}\otimes x_{2}$.
This gives us $(ii).$
\end{proof}
\comments{
\begin{rem}\label{side pat len 1}
A closer look at the proof of \Cref{t type}(1) gives the exact form of $ v $ for which $  \mcal A\Lambda_{\phi,w}^{v} $ is non-zero for $\abs w=1 $. 
\end{rem}
\begin{rem}\label{non t-type}
From Case 2 in the proof of \cref{t type}(2)(b)(ii), it follows that if at least one of $ w_1,w_2 $ is not in Ad($ \mcal C $) (resp. Ad($ \mcal D $)), then $ \mcal A\Lambda_{w_1,w_2}= \mcal A\mcal C_{w_1,w_2} $ (resp. $ \mcal A\Lambda_{w_1,w_2}= \mcal A\mcal D_{w_1,w_2} $).	
\end{rem}
}
\section{Annular representations of free product of categories}

Let $\mcal{C}$ be an arbitrary rigid $ C^* $-tensor category, and $\Gamma\subseteq [\t{Obj }\mcal{C}]$ be an arbitrary weight set containing $\mathbbm{1}$, which is sufficiently full to generate a universal C*-algebra.
Consider the ideal  $\mcal{J}\Gamma_{0}:=\mcal{A}\Gamma\cdot \mcal{A}\Gamma_{\mathbbm{1},\mathbbm{1}}\cdot \mcal{A}\Gamma$
in $ \mcal A\Gamma $ generated by $ \mcal A\Gamma_{\mathbbm{1},\mathbbm{1}} $.
In the particular case of $ \Gamma = \Irr(\mcal C) $, we write $ \mcal J\mcal C_0 $ for $ \mcal{J}\Gamma_{0} $.

Any bounded $*$-representation of $\mcal{J}\Gamma_{0}$ defines a bounded $*$-representation of $\mcal{A}\Gamma$.
In fact, the induction functor $Ind_{0}:Rep(\mcal{J}\Gamma_{0})\rightarrow Rep(\mcal{A}\Gamma)$ is a fully faithful functor, and its image defines the full subcategory $Rep_{0}(\mcal{A}\Gamma)$ of representations generated by their weight $\mathbbm{1}$ space.
Furthermore, $Rep_{0}(\mcal{A}\Gamma)$ is precisely the category of \textit{admissible representations of the fusion algebra} with respect to $\Gamma$.

Consider the W*-category $Rep_{+}(\mcal{A}\Gamma):=Rep(\mcal{A}\Gamma / \mcal{J}\Gamma_{0})$ of representations of $\mcal{A}\Gamma$ which contain $\mcal{J}\Gamma_{0}$ in their kernel.
$ Rep_+(\mcal A\Gamma) $ is referred to as the category of \textit{higher weight representations}.
It consists of precisely the representations of $\mcal{A}\Gamma$ such that the projection $p_{\mathbbm{1}}\in \mathcal{A}\Gamma_{\mathbbm{1},\mathbbm{1}}$ acts by $0$.

Then, for any non-degenerate $*$-representation of $(\pi, \mcal H)\in Rep(\mcal{A}\Gamma)$, we can decompose $\mcal H$ as direct sum of subrepresentations $\mcal H_{0}\oplus \mcal H^{\perp}_{0}$, where $\mcal H_{0}:=[\pi(\mcal J\Gamma_{0})\mcal H]$ and $\mcal H^{\perp}_{0}$ is its orthogonal complement.
We can view $\mcal H_{0}\in Rep_{0}(\mcal{A}\Gamma)$ and $\mcal H^{\perp}_{0}\in Rep_{+}(\mcal{A}\Gamma)$.
Any representation of $\mcal{J}\Gamma_{0}$ and any representation of $\mcal{A}\Gamma_{+}$ are disjoint as representations of $\mcal{A}\Gamma$.
This discussion gives us the following proposition:

\begin{prop} For any sufficiently full weight set, $Rep(\mcal{A}\Gamma)\cong Rep_{0}(\mcal{A}\Gamma)\oplus Rep_{+}(\mcal{A}\Gamma)$.
\end{prop}

Thus, the problem of understanding $Rep(\mcal{A}\Lambda)$ decomposes into the problem of understanding the admissible representations of the fusion algebra, and the higher weight structure.
In the particular case of free products, what we will see is that the weight $0$ part is controlled by a free product C*-algebra, while the higher weight parts can be read off in terms of the higher weight parts of $\mcal{C}$ and $\mcal{D}$.  There are also some additional copies of the category $Rep(\Z)$ that appear at higher weights.

We first turn our attention to the weight $0$ case.
Let $\text{Fus}(\mcal{C})$ be the fusion algebra of $\mcal{C}$ with the distinguished basis $  \Irr(\mcal{C}) $.
Recall there exists a universal C*-algebra completion of the fusion algebra, denoted by $C^{*}_{u}(\mcal{C})$, first introduced by Popa and Vaes \cite{PV}, which is universal with respect to \textit{admissible representations}.
In \cite{GJ}, it was shown that $\mcal{A}\mcal{C}_{\mathbbm{1},\mathbbm{1}}\cong \text{Fus}(\mcal{C})$ and admissible representations are precisely those that induce bounded $*$-representations of the tube algebra, and thus $C^{*}_{u}(\mcal{C})$ can be viewed as the weight $0$ corner (or centralizer algebra) of the universal C*-algebra of the tube algebra. 

Via the inclusions of $\mcal{C}$ and $\mcal{D}$ into $\mcal{C}\ast\mcal{D}$, $ \text{Fus}\left(\mcal{C}\ast\mcal{D}\right) $ contains the fusion algebras $ \text{Fus}(\mcal{C}) $ and $\text{Fus}(\mcal{D})$ as unital $ * $-subalgebras.
Indeed, we have a canonical $ * $-algebra isomorphism between $\text{Fus}(\mcal{C}\ast\mcal{D})$ and the (algebraic) free product $\text{Fus}(\mcal{C})\ast\text{Fus}(\mcal{D})$.

We briefly recall the definition of (universal) free product of C*-algebras:

\begin{defn}
If $A_{1}$ and $A_{2}$ are unital C*-algebras, a \textit{free product} is a unital C*-algebra $A_{1}\ast A_{2}$, together with unital $*$-homomorphisms $\iota_{i}:A_{i}\rightarrow A_{1}\ast A_{2}$ satisfying the following universal property: for any unital C*-algebra $ C $ and unital $*$-homomorphisms $\gamma_{i}:A_{i}\rightarrow C$ there exists a unique $*$-homomorphism $\gamma_{1}\ast\gamma_{2}:A_{1}\ast A_{2}\rightarrow C$ such that $\left(\gamma_{1}\ast\gamma_{2}\right)\circ \iota_{i}=\gamma_{i}$.

\end{defn}

Any two free products of two C*-algebras are $ * $-isomorphic if they exists by the universal property.
Furthermore, free products \textit{do} exist, so it makes sense to talk about \textit{the} free product C*-algebra, which we will denote by $A_{1}\ast A_{2}$.

The main result of this section is the following:

\begin{prop}\label{freeprodweight0}
$C^{*}_{u}(\mcal{C}\ast\mcal{D})\cong C^{*}_{u}(\mcal{C})\ast C^{*}_{u}(\mcal{D})$.
\end{prop}

To prove this, we already know that $ \mcal A \mcal C_{\mathbbm{1} , \mathbbm{1}} $, $ \mcal A \mcal D_{\mathbbm{1} , \mathbbm{1}} $ and $ \mcal A \Lambda_{\emptyset,\emptyset} $ are isomorphic to the fusion algebras $ \text{Fus}(\mcal{C}) $, $ \text{Fus}(\mcal{D}) $ and $\text{Fus}(\mcal{C}\ast \mcal{D})\cong \text{Fus}(\mcal{C})\ast\text{Fus}(\mcal{D})$ respectively.
Using these isomorphisms, any representation of the weight zero centralizer algebra $ \mcal A \Lambda_{\emptyset, \emptyset} $ can also be viewed as representations of $ \mcal A \mcal C_{\mathbbm{1} , \mathbbm{1}} $ and $ \mcal A \mcal D_{\mathbbm{1} , \mathbbm{1}} $ by restricting $ \pi $ to the corresponding subalgebras. 
We have the following lemma:

\begin{lem}\label{freeadmis}
A representation $ (\pi, \mcal H) $ of  $\text{Fus}(\mcal{C}\ast \mcal{D})$
% \cong \mcal F_e \cong \mcal F_c \ast \mcal F_d 
is admissible if and only if its restrictions $ (\pi^c,\mcal H) $ and $ (\pi^d,\mcal H) $ to $\text{Fus}(\mcal{C}) $ and $\text{Fus}(\mcal{D})$ are admissible respectively.
\end{lem}
\begin{proof}
If $ (\pi, \mcal H) $ be admissible then, $ (\pi^c,\mcal H) $ and $ (\pi^d,\mcal H) $ are clearly admissible.
	
Suppose  $ (\pi^c,\mcal H) $ and $ (\pi^d,\mcal H) $ are admissible.
Set $ \mcal {\widehat H}_w:=\mcal A\Lambda_{\emptyset,w} \otimes \mcal H $ for $ w \in \Lambda $.
By Lemma \ref{uneq len} and Lemma \ref{adjointweight0}, $ \widehat{\mcal H}_w $ is nonzero only when $ w = \emptyset $  or $ w $ has length $ 1 $ and is in $ \textbf{S}(\mcal{C})\cup \textbf{S}(\mcal{D}) $.
As usual, we define a sesquilinear form on $ \widehat{\mcal H}_w $ by 
$$ \langle y_1\otimes\xi_1, y_2\otimes\xi_2\rangle_w:=\langle\pi(y_2^{\#}\cdot y_1)\xi_1,\xi_2\rangle $$
for $ y_1,y_2 \in \mcal A\Lambda_{\emptyset,w} $ and $ \xi_1,\xi_2 \in \mcal H $.
%By \Cref{high wt}, we just need to consider $ w\in \mcal I_c\cup\mcal I_d $.
%Further, without loss of generality, assume $ w \in \mcal I_c $.
%By \Cref{t type}, $ w\in \t{Ad}(\mcal C) $.

By the definition of admissibility and \cite{GJ}, it suffices to show that this form is positive semi-definite.
Further, it is enough to show 
$$\sum\limits_{i,j=1}^{n} \langle\pi(x_{j}^{\#}\cdot x_i)\xi_i,\xi_j\rangle \geq 0$$
for $ x_i \in \mcal  A\Lambda^{v_i}_{\emptyset,w} $, $ v_i \in \Irr(\mcal{C}*\mcal{D}) $, $ \xi_i \in \mcal H $.
%,  $ u \in \C[\mcal I_e] $ and $ v_i \in \mcal I_c $ are such that $ \mcal C(v_i,wv_i) \neq \{0\} $,  for each $ i=1,2,\ldots n $.
When $ w = \emptyset $, the sum becomes $ \sum \limits^n_{i=1} \norm{\pi(x_i) \xi_i}^{2}_{\mcal H} \geq 0 $.
It remains to consider the case $ w \in \textbf{S}(\mcal{C})\cup \textbf{S}(\mcal{D}) $.
Suppose $w\in \textbf{S}(\mcal{C})$.
In order to have $ \mcal A \Lambda^{v_i}_{\emptyset,w}  = \left(\mcal{C}*\mcal{D}\right) (v_i , wv_i)$ nonzero, $ v_i $  must be one of $ \mcal{C} $-$ \mcal{C} $ or $ \mcal{C} $-$\mcal{D}$ type.
Let $ v_i = c_i u_i $ where $ c_i \in \IC$ and $ u_i $ is either $ \emptyset $ or of $ \mcal{D} $-$ \mcal{C} $ or $ \mcal{D} $-$ \mcal{D} $ type.
Note that $ wv_i = wc_i u_i $.
As $w\in \mcal{C}$, any morphism $x_{i}\in \left(\mcal{C}*\mcal{D}\right)( c_i u_i, wc_i u_i  )$ is of the form $x_{i}=z_{i}\otimes 1_{u_i}$, where $z_{i}\in \mcal{C}(c_{i}, wc_{i})$.

One may express this in another useful way: $ x_i = z_i \cdot 1_{u_i} $ where we view $ z_i \in \mcal A \mcal C^{c_i}_{\mathds 1,w} \subset \mcal A \Lambda_{\emptyset,w}$, and $ 1_{u_i} \in \mcal A \Lambda^{u_i}_{\emptyset,\emptyset}$.
Setting $ \zeta_i := \pi(1_{u_i}) \xi_i $, $ 1\leq i \leq n $, we have
\[
\sum\limits_{i,j=1}^{n} \langle\pi(x_{j}^{\#}\cdot x_i)\xi_i,\xi_j\rangle = \sum\limits_{i,j=1}^{n} \langle\pi^c(z_{j}^{\#}\cdot z_i)\zeta_i,\zeta_j\rangle\ge 0
\]
where the last inequality follows from admissibility of $ (\pi^c , \mcal H) $.
An entirely analogous argument holds for the case $w\in \textbf{S}(\mcal{D})$.
\comments{Let $ \mcal {\wt H}_w $ be the completion of the quotient of $ \mcal {\widehat H}_w $ over the null space of the form.
Define the action of $ z \in \mcal A\Lambda_{w_1,w_2} $ by $ \wt\pi(x)[x\otimes\xi] = [zx\otimes\xi] $ where $ x \in \mcal A\Lambda_{\emptyset,w_1} $ and $ \xi \in \mcal H $.
Boundedness of $ \widetilde{\pi} (x) $ and well-definedness and $ * $-homomorphism property of the action $ \widetilde \pi $ follows from \red{[GJ] enough or more?}.
Finally, $ \widetilde H_\emptyset \ni [x \otimes \xi] \longmapsto \pi(x)\xi \in \mcal H $ extends to an isometric isomorphism which intertwines the action of $ \left. \widetilde \pi  \right|_{\mcal A \Lambda_{\emptyset, \emptyset}} $ and $ \pi $.}
\end{proof}
\medskip

\begin{proof}[Proof of Proposition \ref{freeprodweight0}]
Let $ i_{\mcal{C}} $ (resp., $ i_{\mcal{D}} $) be the canonical $*$-inclusion of $ \text{Fus}(\mcal{C}) $ (resp., $ \text{Fus}(\mcal{D})$) into $ \text{Fus}(\mcal{C}\ast \mcal{D}) $.

If $(\pi,\mcal H)$ is any admissible representation of $\text{Fus}(\mcal{C}\ast \mcal{D})$, then $(\pi\circ i_{\mcal{C}}, \mcal H)$ and $(\pi\circ i_{\mcal{D}}, \mcal H)$ are admissible representations of $\text{Fus}(\mcal{C})$ and $\text{Fus}(\mcal{D})$ respectively by Lemma \ref{freeadmis}.  Therefore, for any $x\in \text{Fus}(\mcal{C})$, 
$$||i_{\mcal{C}}(x)||_{\pi}=||x||_{\pi\circ i_{\mcal{C}}}\le ||x||_{C^{*}_{u}(\mcal{C})} \ .$$  
By the definition of the universal norm, 
$$||i_{\mcal{C}}(x)||_{C^{*}_{u}(\mcal{C}*\mcal{D})}=\sup_{\pi^{\prime}} ||i_{\mcal{C}}(x)||_{\pi^{\prime}}$$
where the supremum is taken over all admissible representations of $\text{Fus}(\mcal{C}\ast\mcal{D})$.
Thus the map $i_{\mcal{C}}$ extend to $*$-homomorphisms $\iota_{\mcal{C}}: C^{*}_{u}(\mcal{C})\rightarrow C^{*}_{u}(\mcal{C}\ast\mcal{D})$.
The same argument applies to $\mcal{D}$, yielding an extension $\iota_{\mcal{D}}: C^{*}_{u}(\mcal{D})\rightarrow  C^{*}_{u}(\mcal{C}\ast\mcal{D})$.

Let $ A $ be any C*-algebra with *-homomorphisms $ \gamma_{\mcal{C}} : C^{*}_u(\mcal{C}) \rightarrow A$ and $ \gamma_{\mcal{D}} : C^{*}_u(\mcal{D}) \rightarrow A  $.
By the universal property of free product of ordinary $ * $-algebras, there is a unique $ * $-homomorphism $ h : \text{Fus}(\mcal{C}\ast\mcal{D}) \rightarrow A$ such that $h\circ i_{\mcal{C}}= \gamma_{\mcal{C}}|_{\text{Fus}(\mcal{C})}$ and $h\circ i_{\mcal{D}}= \gamma_{\mcal{D}}|_{\text{Fus}(\mcal{D})}$.
By density of the fusion algebras in their universal C*-algebras, to conclude the proof it suffices to show that $h$ extends to a $*$-homomorphism $\gamma_{\mcal{C}}\ast\gamma_{\mcal{D}}:C^{*}_{u}(\mcal{C}\ast\mcal{D})\rightarrow A$, which is equivalent to showing $||h(x)||_{A}\le ||x||_{C^{*}_{u}(\mcal{C}*\mcal{D})}$.

Without loss of generality, assume $ A \subset B(\mcal K) $ for some Hilbert space $ \mcal K $.
Since $||\gamma_{\mcal{C}}(y)||_{A}\le||y||_{C^{*}_{u}(\mcal{C})}$ for every $y\in \text{Fus}(\mcal{C})$, $ (\gamma_{\mcal{C}}|_{\text{Fus}(\mcal{C})}, \mcal K) $ is admissible and similarly, $ (\gamma_{\mcal{D}}|_{\text{Fus}(\mcal{D})}, \mcal K) $ is also admissible.
Thus, by Lemma \ref{freeadmis}, $(h,\mcal K)$ is an admissible representation of $\text{Fus}\left(\mcal{C}*\mcal{D}\right)$.
Therefore, $||x||_{h}=||h(x)||_{A}\le ||x||_{C^{*}_{u}(\mcal{C}*\mcal{D})} $.
% at the level of fusion algebras and  $ C^*_u(\mcal F_e)\cong C^*_u(\mcal F_c)\ast C^*_u(\mcal F_d) $. 
\end{proof}
%%%%%%%%%%%%%%%%%%%%%%%%%%%%%%%%%%%%%%
%%%%%%%%%%%%%%%%%%%%%%%%%%%%%%%%%%%%%%
This immediately implies the following corollary:

\begin{cor}  The category of $Rep_{0}(\mcal{A}\Lambda)$ is equivalent as a W*-category to $Rep(C^{*}_{u}(\mcal{C})\ast C^{*}_{u}(\mcal{D}))$.
\end{cor}

On one hand, it is well known that representation categories of free product algebras are wild and uncontrollable, and thus this answer for describing $Rep_{0}(\mcal{A}\Lambda)$ is somewhat unsatisfactory, compared to descriptions of other representation categories such as $Rep(\mcal{A}TLJ)$ (\cite{GJ}).
On the other hand, there are a plethora of ways to produce examples of representations of free products, so these categories are quite flexible.
For example, given two states $\psi, \phi$ on C*-algebras $A$ and $B$, one can construct the free convolution state $\psi*\phi$ on the C*-algebra $A*B$ (\cite{A}).
Alternatively one simply has to take a representation of $A$ and one of $B$, and identify their underlying Hilbert space.

We now move on to describing the higher weight categories, which, depending on $\mcal{C}$ and $\mcal{D}$, can be more manageable.
As described in the beginning of the section $Rep_{+}(\mcal{A}\Lambda)=Rep(\mcal{A}\Lambda/\mcal{J}\Lambda_{0})$.
We have the following lemma:

\begin{lem} 
As $*$-algebras, $\mcal{A}\Lambda / \mcal{J}\Lambda_{0}\cong \mcal{A}\mcal{C} /\mcal{J}\mcal{C}_{0} \oplus \mcal{A}\mcal{D} /\mcal{J}\mcal{D}_{0}\oplus \mcal{A}\W$.
\end{lem}

\begin{proof}
Recall that $\mcal{A}\Lambda\cong \mcal{A}[\Lambda\setminus \W]\oplus \mcal{A}\W$.  %Let $\mcal{J}\Lambda_{0}\subseteq  \mcal{A}\Lambda $ be the ideal in $\mcal{A}\Lambda$ generated by $\mcal{A}\Lambda_{\phi,\phi}$.
From Lemma \ref{uneq len}, we see that $\mcal{J}\Lambda_{0}\subseteq \mcal{A}[\Lambda\setminus \W]$, and thus

$$\mcal{A}\Lambda / \mcal{J}\Lambda_{0}\cong \mcal{A}[\Lambda\setminus \W] / \mcal{J}\Lambda_{0}\oplus \mcal{A}\W $$

Thus we consider the spaces $\mcal{A}\Lambda^{v}_{w_{1},w_{2}}$ with $w_{1},w_{2}\in \textbf{S}(\mcal{C})\cup \textbf{S}(\mcal{D})$, and $v\in \Irr(\mcal{C}*\mcal{D})$.  By Lemma \ref{adjointcd} and Lemma \ref{adjointcc}, the image of these spaces under the quotient is $0$ unless $w_{1}$ and $w_{2}$ are either both in $\textbf{S}(\mcal{C})$ and $v\in \Irr(\mcal{C})$ or both $w_{1}$ and $w_{2}$ are in $\textbf{S}(\mcal{D})$ and $v\in \Irr(\mcal{D})$.
Since $\mcal{J}\mcal{C}_{0}, \mcal{J}\mcal{D}_{0}\subseteq \mcal{J}\Lambda_{0}$, it is now clear that the quotient map assembles into an isomorphism  $\mcal{A}[\Lambda\setminus \W] / \mcal{J}\Lambda_{0}\cong \mcal{A}\mcal{C}/ \mcal{J}\mcal{C}_{0}\oplus \mcal{A}\mcal{D}/\mcal{J}\mcal{D}_{0}$, concluding the proof.
\end{proof}

Finally, we recall that $\W_{0}$ is the set of cyclic equivalence classes of words in $\W$, and note that $Rep(\mcal{A}\W)\cong Rep(\Z)^{\oplus \W_{0}}$.
The above results imply Theorem \ref{mainthm}, which is the main result of this article.
\section{Examples}

In this section, we apply the main result to several examples.  First, we show how this matches another known result.

\begin{ex}{\textbf{Free products of group categories}.}
In particular, for any countable group $G$, we consider the rigid C*-tensor category $\Hilb_{f.d.}(G)$ of finite dimensional $G$-graded Hilbert spaces.  Let $\Lambda$ denote the set of conjugacy classes of $G$.  For each $\lambda\in \Lambda$ we can define $C_{\lambda}(G)$ to be the centralizer subgroup of some element $g\in \lambda$.  We note that different choices of $g\in \Lambda$ yield conjugate subgroups, and so $C_{\lambda}(G)$ is well defined up to isomorphism.  Then, from \cite{GJ}, the category of annular representations 

$$Rep(\mcal{A})\cong \bigoplus_{\lambda\in \Lambda} Rep(C_{\lambda}(G)) $$

Now, for any two countable groups $G$ and $H$, its easy to see that $\Hilb_{f.d.}(G)*\Hilb_{f.d.}(H)$ is equivalent as a C*-tensor category to $\Hilb_{f.d.}(G*H)$.  Thus we can compare our result for $\Hilb_{f.d.}(G)*\Hilb_{f.d.}(H)$ to the above result for $\Hilb_{f.d.}(G*H)$.

Since $C^{*}_{u}(\Hilb_{f.d.}(G))$ is isomorphic to the universal group C*-algebra $C^{*}_{u}(G)$, and $C^{*}_{u}(G*H)\cong C^{*}_{u}(G)*C^{*}_{u}(H)$, we can identify the first component in the main theorem (Theorem \ref{mainthm}) with $Rep(G*H)$.

Note that there is always distinguished conjugacy class $[1]\in \Lambda$, the conjugacy class of the unit $1$.  We have $C_{[1]}(G)=G$.  It is easy to see that 
$$Rep_{+}(\mcal{A}\Hilb_{f.d.}(G))\cong \bigoplus_{\lambda\in \Lambda\setminus [1]} C_{\lambda}(G)$$
This helps us identify the second two components, while the last component needs no identification.

Now, consider the group $G*H$.  This group has 4 types of conjugacy classes:  $\{[1]\}, \{[g]\ :\ g\in G\},\{[h]\ :\ h\in H\}$ and $\{[g_{1}h_{1}\cdots g_{k}h_k]\ :\ g_{i}\in G, h_{i}\in H,\ k\ge 1\}$.  It is also easy to see that $C_{[1]}(G*H)=G*H$, $C_{[g]}(G*H)=G$, $C_{[h]}(G*H)=H$ and $C_{[g_{1}h_{1}\cdots g_{k}h_{k}]}=\{(g_{1}h_{1}\cdots g_{k}h_{k})^{n}\ :\ n\in \Z\}\cong \Z$.  It is now easy to see the equivalence of the two descriptions.

\end{ex}

\begin{ex}{\textbf{Fuss-Catalan representations}.}
Bisch and Jones introduced the \textit{Fuss-Catalan} subfactor planar algebras $\mathcal{FC}(\alpha,\beta)$, where $\alpha,\beta\in \{2\cos(\frac{\pi}{n})\ :\ n\ge 3\}\cup [2,\infty)$ \cite{BJ}.
These planar algebras are universal for intermediate subfactors.
For a subfactor planar algebra, the category of affine annular representations in the sense of Jones-Reznikoff \cite{JR} is equivalent to the category of annular representations of the even part of the subfactor (see, for example, \cite[Remark 3.6]{DGG}  or \cite[Corollary 4.4]{NY2}).  The even part of the Fuss-Catalan can be realized as a full subcategory of the free product category $\mathcal{TLJ}(\alpha)*\mathcal{TLJ}(\beta)$.
In particular, if $a\in \mathcal{TLJ}(\alpha)$ is the standard tensor generating object with dimension $\alpha$ and $b\in \mathcal{TLJ}(\beta) $ is the standard tensor generating object with dimension $\beta$, then the full subcategory generated by $abba\in \mathcal{TLJ}(\alpha)*\mathcal{TLJ}(\beta)$ is equivalent to the even part of $\mathcal{FC}(\alpha,\beta)$.
Thus to determine the annular representation category of $\mathcal{FC}(\alpha,\beta)$, it suffices to determine the annular representations of the full subcategory $\mathcal{TLJ}(\alpha)*\mathcal{TLJ}(\beta)$ generated by $abba$.
Let $\mathcal{TLJ}_{0}(\alpha)$ denote the adjoint subcategory, generated by $aa$.
This can also be realized as the even part of the usual Temperley-Lieb-Jones subfactor planar algebras.

We recall briefly that two rigid C*-tensor categories $\mcal{C}$ and $\mcal{D}$ are \textit{weakly Morita equivalent} if there is a rigid C*-2 category with two objects $0$ and $1$, such that the tensor category $\text{End}(0)\cong\mcal{C}$ and the tensor category $\text{End}(1)\cong \mcal{D}$ (see \cite{NY2} for further details).
The two even parts of a subfactor planar algebra are weakly Morita equivalent, but weak Morita equivalence is more general.
If we have two full subcategories of a tensor category, to show they are weakly Morita equivalent, it suffices to find an object $x\in \mcal{C}$ so that $x\overline{x}$ tensor generates one and $\overline{x}x$ tensor generates the other, since one can, using the usual subfactor approach, construct a rigid C*-2 category whose two even parts are as desired.
We apply this in the free product case to obtain the following proposition:

\begin{prop}
The tensor category generated by $abba$ is weakly Morita equivalent to $\mathcal{TLJ}_{0}(\alpha)*\mathcal{TLJ}_{0}(\beta)$.
\end{prop}

\begin{proof}
It suffices to find an object $x\in \mathcal{TLJ}(\alpha)*\mathcal{TLJ}(\beta) $ such that $\langle x\overline{x}\rangle =\langle abba\rangle $ and $\langle \overline{x} x\rangle=\mathcal{TLJ}(\alpha)*\mathcal{TLJ}(\beta)$.
Choose $x:=abb$.
Then since both $aa$ and $bb$ contain the tensor unit as a subobject, we see $\langle abbbba\rangle=\langle abba\rangle$.
On the other hand, $bbaabb$ contains $aa$ and $bb$ as subobjects, and so clearly $\langle bbaabb\rangle=\langle aa, bb\rangle$.
\end{proof}

Again, by \cite[Remark 3.6]{DGG} or \cite[Corollary 4.4]{NY2}, the above proposition implies the following:

\begin{cor}
The category of affine annular representations of the subfactor planar algebra $\mathcal{FC}(\alpha,\beta)$  is equivalent as a W*-category to the annular representation category of $\mathcal{TLJ}_{0}(\alpha)*\mathcal{TLJ}_{0}(\beta)$.
\end{cor}

This category $\mathcal{TLJ}_{0}(\alpha)$ is fully described in \cite{JR}, and thus combining those results with ours leads to a description of the representations of Fuss-Catalan categories.

\end{ex}
\bibliographystyle{alpha}

%\bibliographystyle{alpha}
%\bibliography{ref}
\end{document}